\documentclass[12pt,twoside]{amsart} 
\usepackage{amssymb,color,tikz}
\usetikzlibrary{matrix} 

\usepackage{enumitem}
\usepackage{pgf,tikz}
\usetikzlibrary{arrows}

\usepackage{xcolor}

\usepackage{cleveref}    

\usepackage{comment}
\usepackage{faktor}

\numberwithin{equation}{section} 

\newcommand {\CC}{\mathbb{C}}
\newcommand {\PP}{\mathbb{P}}
 
\DeclareMathOperator{\codim}{codim} 

\DeclareMathOperator{\Exc}{Ex.\!C}
\DeclareMathOperator{\Jac}{Jac}

\DeclareMathOperator{\Syz}{Syz} 

\def\cocoa{{\hbox{\rm C\kern-.13em
      o\kern-.07em C\kern-.13em o\kern-.15em A}}}

\newtheorem{theorem}{Theorem}[section]
\newtheorem{lemma}[theorem]{Lemma}
\newtheorem{proposition}[theorem]{Proposition}
\newtheorem{corollary}[theorem]{Corollary}
\newtheorem{conjecture}[theorem]{Conjecture} \theoremstyle{definition}
\newtheorem{definition}[theorem]{Definition} 
\newtheorem{remark}[theorem]{Remark}
\newtheorem{example}[theorem]{Example}

\newtheorem{question}[theorem]{Question}
\definecolor{MyDarkGreen}{cmyk}{0.7,0,1,0}

\newcommand{\adim}{\ensuremath{\mathrm{adim}}}
\newcommand{\edim}{\ensuremath{\mathrm{edim}}}
\newcommand{\vdim}{\ensuremath{\mathrm{vdim}}}

\begin{document}

\title[Unexpected Hypersurfaces and their Consequences]%
{Unexpected hypersurfaces and their consequences: \\ A Survey}

\author[B. Harbourne]{Brian Harbourne$^{*}$}
\address{Department of Mathematics, University of Nebraska, Lincoln, NE 68588-0130}
\email{brianharbourne@unl.edu}

\author[J.\ Migliore]{Juan Migliore${}^{**}$}
\address{ Department of Mathematics, University of Notre Dame, Notre
  Dame, IN 46556, USA} 
\email{migliore.1@nd.edu} 

\author[U.\ Nagel]{Uwe Nagel${}^{***}$} 
\address{Department of
  Mathematics, University of Kentucky, 715 Patterson Office Tower,
  Lexington, KY 40506-0027, USA} \email{uwe.nagel@uky.edu}
  
\thanks{Printed \today\\
Harbourne was partially supported by Simons Foundation grant \#524858. \\
  Migliore was partially supported by  Simons Foundation grant 
  \#839618.\\
  Nagel was partially supported by Simons Foundation grant \#636513.  \\
}

\begin{abstract}
The notion of an unexpected curve in the plane was introduced in 2018, and was quickly generalized in several directions in a flurry of mathematical activity by many authors. In this expository paper we first describe some of the main results on unexpected hypersurfaces. Then we summarize two offshoots of this theory. First we look at sets of points in $\mathbb P^3$ whose general projection is a planar complete intersection (so-called {\it geproci} sets). Although we now  know a lot  about these sets, much remains mysterious. Then we describe an interesting measure of unexpectedness called {\it $AV$-sequences}, which have a surprising structure that is not yet fully understood.
\end{abstract}
\maketitle

\section{Introduction}

Although some of the concepts we discuss here are of interest over any algebraically closed field,
for this exposition we fix the ground field $K$ to be algebraically closed of characteristic zero. 
Our interest here is to bring to the reader's attention a burgeoning but only recently established
research field, and to sketch out  some of the directions into which it has already branched.

The notion of unexpectedness in its current form dates back only to \cite{CHMN},
but it grew out of the much older tradition of studying multivariate polynomial interpolation problems. 
A simple example of a still open interpolation problem is this: given integers $m,n,s\geq 2$, 
what is the least degree $t$ of a nontrivial form $F\in K[\PP^n]$ such that $F$ and all its partials
of order up to $m-1$ vanish on $s$ general points?

An answer is known for $n=2$ and $s<10$. The SHGH Conjecture asserts in the case
of $n=2$ and $s\geq 10$ that the answer is the least $t$ such that
$$\binom{t+2}{2} > s\binom{m+1}{2}.$$
The idea is to show, when one expects by the naive dimension count $\binom{t+2}{2} - s\binom{m+1}{2}$ 
for the dimension of the vector space of such forms that
no form exists, that one's expectations are met; i.e., that no unexpected curves occur.
However, in spite of decades of attention, this remains open, and no conjecture even exists
in general for $n>3$ (see \cite{LU} for a conjecture for $n=3$).

The problem is perhaps that we do not understand well enough how unexpected behavior 
might occur in order to show it does not occur. Thus it could be productive to broaden one's perspective
in such a way that unexpectedness can arise, and by studying its occurrences, understand
better why it seems not to occur in the open problem mentioned above.

This is part of the motivation for the work initiated in \cite{CHMN}. But it turns out that the new notion of unexpectedness
has substantial intrinsic interest, with connections that touch a broad array of disciplines,
including commutative algebra, combinatorics and representation theory, in addition to algebraic geometry.
Here is some of the recent work on this topic \cite{BMSS, DMO, DHRST, FGHM, FGHM, HMNT, MPT, PSS, Szp1, Szp2, trok, WZ}.

More specifically, the idea of unexpectedness initiated in \cite{CHMN} was to say
that $Z\subset \PP^2$ had unexpected curves 
of degree $t$ with respect to a general point of multiplicity $t-1$ if
the vector space $[I_Z]_t\cap [I_P^{t-1}]_t\subset K[\PP^2]$ of forms of degree $t$
vanishing at $Z$ and vanishing to order $t-1$ at a general point $P$ had a positive dimension bigger than
one would expect by a naive dimension count. In other words,
that $\dim [I_Z]_t\cap [I_P^{t-1}]_t > \dim [I_Z]_t-\binom{t}{2}$.
The goal was to understand which finite sets of points had such unexpected curves.

One of the reasons this has been so fruitful is that examples do arise, and they often involve very interesting structure.
The first example we became aware of came from \cite{DIV} in work related to Lefschetz problems
and it involved the $B_3$ root system. (This connection to Lefschetz problems and to root systems has persisted and
is one of the ways unexpectedness touches upon both commutative algebra and representation theory.)
Looking at the lines dual to the points has also been very useful \cite{DIV, FV, CHMN, FGST, BMSS, DMO}.
(This brings this research into contact with combinatorics in the guise of line arrangements. Indeed, one 
of the ways unexpectedness often
seems to arise is from points dual to complex line arrangements which have no points where exactly two lines intersect,
but only a few kinds of such arrangements are currently known. It is an open problem to find more examples
of such arrangements, or to show no examples beyond those currently known exist.)

The first two natural extensions of this idea are to move the problem to hypersurfaces in higher dimensional projective spaces, and to relax the restriction that the degree of the hypersurface should be one more than the multiplicity. As a first step,   
it has also been very productive to look at sets of points coming from other root systems, as was initiated in \cite{HMNT} (and continued in such work as \cite{HMT, trok}).  However, the theory has gone far beyond root systems. A third extension has been studied by Trok in \cite{trok}. He considers unexpected hypersurfaces that vanish on  a general linear linear subspace with prescribed multiplicity and 
an arbitrary finite set of points. In particular, Trok introduces the concept of a very unexpected hypersurface, which allows him to generalize the connection to hyperplane arrangements from $\PP^2$ to $\PP^n$. 
Some of the relevant research is described in \S\ref{uwe}.

The notion of unexpected hypersurfaces, especially the case of cones (where the degree and the multiplicity coincide) led to a surprising new notion,  
that of $(a,b)$-geproci sets (i.e., finite sets in $\PP^3$ whose projection to a plane from a general point is a complete intersection of type $(a,b)$ 
\cite{CM, WZ, PSS, CFFHMSS}). There is a simple example of such sets for any choice of $a,b$  called {\it grids}, but finding non-grid examples is highly non-trivial. This problem has been studied extensively especially in \cite{CFFHMSS}, and the current state of the art is described in \S\ref{brian}. 

There is another offshoot of the study of unexpected hypersurfaces, namely the notion of {\it $AV$-sequences}. These sequences were created as a measure of the extent of unexpectedness admitted by a given scheme (and now we no longer insist that we restrict to zero-dimensional schemes). In some sense they measure the {\it persistence} of unexpectedness. Some very surprising patterns have emerged, and many of the known results (and an intriguing conjecture) are described in \S\ref{juan} and \S\ref{juan2}.

Throughout this paper, the abbreviation {\it ACM} will stand for {\it Arithmetically Cohen-Macaulay}. We recall that a scheme $X$ is ACM if $R/I_X$ is a Cohen-Macaulay ring, and that there are cohomological and homological criteria for this property that are standard and not reviewed here.


\section{Unexpected curves and hypersurfaces} \label{uwe}

A basic problem in algebraic geometry is to determine the dimension of a linear system on a projective variety. For example, various versions of the Riemann-Roch theorem may be viewed as tools to address this challenge.  It is typically the case that there is an {\em expected} dimension (or codimension), given by naively counting constants, and that the expected dimension equals the actual dimension in most cases. Understanding then the  {\em special} linear systems, i.e., those whose actual dimensions are greater than the expected ones, is a subtle problem of substantial interest. 

As an example, consider the complete linear system $\mathcal L_j$ of projective plane curves of degree $j$. Its vector space dimension is  $\binom{j+2}{2}$. For integers $j\geq m$,  the requirement that the curves all have multiplicity at least $m$ at a fixed point $P$ imposes $\binom{m+1}{2}$ linear conditions. Indeed,  the   linear system of all such curves has dimension $\binom{j+2}{2} - \binom{m+1}{2}$, and so the actual and expected dimensions coincide. We will refer to this system as the  linear subsystem of curves passing through a {\em fat point of multiplicity $m$}  supported at $P$. It is a much-studied, but still open problem to compute the dimension of the linear system $\mathcal L_j$ of curves of degree $j$ passing through a {\em general} set of $r$ fat points $P_1,\ldots,P_r$ with multiplicities $m_1,\dots, m_r$. Write $X = m_1 P_1 + \cdots + m_r P_r$ for the fat point scheme supported on  $P_1,\ldots,P_r$. Thus $X$ is defined by the ideal
\[
I_X = I_{P_1}^{m_1} \cap \cdots \cap I_{P_r}^{m_r} \subset K[x, y, z], 
\]
and, abusing notation,  $\mathcal L_j = [I_X]_j$. Its expected dimension is 
\[
\max \left \{0, \ \binom{j+2}{2}  - \sum_{i=1}^r \binom{m_i+1}{2} \right \}. 
\]
The still open {\em SHGH Conjecture} \cite{segre, Ha1, G, Hi} posits a complete list of all 
$(m_1,\ldots,m_r)$ and $j$ for which the actual dimension of  $[I_X]_j$ is larger than the expected dimension and predicts the actual dimension in these cases. 

It is not too difficult to show that the expected and the actual dimension are the same if $m_1 = \cdots = m_r = 1$ and the points $P_1,\ldots,P_r$ are general. Beginning in \cite{DIV, CHMN}, a variation of the above linear systems has been studied. Let $Z \subset \PP^2$ be a finite set of points and $P$ a general point of $\PP^2$. Thus, the linear system of degree $j$ plane curves passing through $Z$ and having multiplicity at least $m$ at $P$ is given by  $[I_X]_j = [I_Z \cap I_P^m]_j$, where $X = Z + mP$. Its actual dimension is 
\[
{\adim} (Z,j,m) = \dim [I_Z \cap I_P^m]_j, 
\]
whereas its expected dimension is 
\[
{\edim} (Z,j,m) = \max \left  \{ 0, \dim [I_Z]_j - \binom{m+1}{2} \right \}. 
\]
An example \cite[Proposition 7.3]{DIV} shows the surprising fact that already in this case, it is no longer true that the expected dimension is necessarily achieved if $Z$ is {\em not} assumed to be a general set of points. Thus, the problem becomes to understand how the geometry of $Z$ can affect the desired dimension. A systematic study of this question has been initiated in \cite{CHMN}. In particular, it is shown that it has  interesting connections to the theory of line arrangements. For example, the results in \cite{CHMN} give  new perspectives on Terao's freeness conjecture, including a generalization to non-free arrangements. 

We now describe some of the results in \cite{CHMN} in more detail. Recall that $X = Z + mP$, where $Z \subset \PP^2$ is an arbitrary non-empty finite set of points and $P \in \PP^2$ a general point. 
Observe that in any degree $j <  m$ one has $\dim [I_X]_j = 0$ and if $j \ge m$ then 
$\dim [I_X]_j \geq \dim [I_Z]_j-\binom{m+1}{2}$;
i.e., the forms in $[I_X]_j$ are obtained from those of $[I_Z]_j$ by imposing at most
$\binom{m+1}{2}$ linear conditions coming from $mP$. Moreover, if $j = m$, then the last estimate is actually always sharp. Thus, the smallest degree in which the actual is not equal to the expected dimension is $j = m+1$. Following \cite[Definition 2.1]{CHMN}, we say that  $Z \subset \PP^2$  {\em admits an unexpected curve}  if there is an integer $j>0$ such that
\begin{equation}
  \label{eq:unexpected}
 \dim [I_{Z+jP}]_{j+1}   > \max \left  \{0, \ \dim [I_Z]_{j+1} - \binom{j+1}{2}\right \}.
\end{equation}
More precisely, in this case we say $Z$ admits an unexpected curve of degree $j+1$. 

It turns out that a set $Z$ of points can have unexpected curves of various degrees. In order to describe 
this range of degrees and to characterize the existence of any unexpected curve several invariants are needed. The \emph{multiplicity index of $Z$} is the integer
\[
m_Z = \min \{ j \in \mathbb Z \ |   [I_{Z+jP}]_{j+1} \neq 0\}
\]
for a general point $P$. It was first considered in work of Faenzi and Vall\'es \cite{FV}. 
Following \cite{CHMN}, the {\em speciality index of $Z$} is 
\[
u_Z = \min \Big \{ j \in \mathbb Z \ | \dim [I_{Z+jP}]_{j+1}=\binom{j+3}{2}-\binom{j+1}{2}-|Z| \Big \}, 
\]
and we define the integer 
\[
t_Z = \min \Big \{ j \in \mathbb Z \ |  \dim [I_Z]_{j+1}>\binom{j+1}{2} \Big \}. 
\]
Note that $t_Z$ depends only on the Hilbert function of $Z$. These three invariants govern the existence of unexpected curves. 

\begin{theorem}[{\cite[Theorem 1.1]{CHMN}}]
     \label{range unexpected curves} 
A finite set $Z \subset \PP^2$ admits an unexpected curve if and only if $m_Z < t_Z$. Furthermore, in this case $Z$ has an 
unexpected curve of degree $j$ if and only if $m_Z < j \le u_Z$. 
\end{theorem}

As shown in \cite{CHMN}, unexpected curves have a particular structure. According to \cite[Corollary 5.5]{CHMN}, if $Z$ admits any unexpected curve it has a unique unexpected curve of degree 
$m_Z + 1$. Denote it by $C_P (Z)$. Any further unexpected curve for $Z$ of degree $t$  is a union of 
$C_P (Z)$ and $t-m_Z-1$ lines, each of them passing through the general point $P$. In particular, any unexpected curve has exactly one irreducible component of degree greater than one. 

\begin{theorem}[{\cite[Theorem 5.9]{CHMN}}]
        \label{thm:unexp curve structure}
Let $Z \subset \PP^2$ be a finite set admitting an unexpected curve and let $P\in\PP^2$ be a general point.
Then there is a unique subset $Z' \subset Z$ such that the unique unexpected degree $m_Z +1$ curve of $Z$,  $C_P(Z)$,  is the union of 
$C_P(Z')$ and $|Z \setminus Z'|$ lines, where $C_P(Z' )$ is irreducible and the unique unexpected 
curve of $Z'$ of degree $m_{Z'}+1$. Furthermore, $C_P(Z')$ is 
rational and smooth away from $P$. 
\end{theorem} 

Note that it follows that any unexpected curve $C$ of $Z$ of degree greater than $m_Z +1$ is always reducible. It is the union of lines and an irreducible curve $C'$, where $C'$ is an unexpected curve of a subset $Z' \subseteq Z$.

The results on the structure of unexpected curves led to the following more geometric version of Theorem \ref{range unexpected curves}:

\begin{theorem}[{\cite[Theorem 1.2]{CHMN}}]
   \label{thm:geomVersion}
Let $Z \subset \PP^2$ be a finite set of points. Then 
$Z$ admits an unexpected curve if and only if
$2m_Z+2<|Z|$ and no subset of $m_Z+2$ (or more) of the points is collinear.
In this case, $Z$ has an 
unexpected curve of degree $j$ if and only if $m_Z < j \le |Z|-m_Z-2$. 
\end{theorem} 

Note that Theorem \ref{range unexpected curves} shows that the existence of an unexpected curve forces $m_Z < u_Z$. 
The converse is almost but not quite true. In fact, \cite[Lemma 3.5(c)]{CHMN} shows that  $2m_Z+2<|Z|$ is equivalent to $m_Z<u_Z$.
Thus $m_Z<u_Z$ together with there being no large collinear subsets of $Z$
implies the occurrence of unexpected curves.

As mentioned above, there is a connection between unexpected curves and line arrangements. Recall that a line arrangement $\mathcal A$ of $\PP^2$ is a union of distinct lines. Suppose that $\mathcal A$ consists of $d$ lines defined by linear forms $\ell_1,\ldots,\ell_d \in K[x,y,z] = R$. Then we refer to $\mathcal A$ as the line arrangement $\mathcal A (f)$ defined by the product $f = \ell_1 \cdots \ell_d$. 

Define the \emph{Jacobian ideal} of $f \in R$ as $\Jac (f) = (f, f_x, f_y, f_z)$. Since we assume that the characteristic of $K$ is zero one has $xf_x+yf_y+zf_z = d f$, and so $\Jac (f) = (f_x, f_y, f_z)$. Thus there is a short exact sequence of graded $R$-modules 
\[
0 \to \Syz(\Jac(f)) \to R^3 (-d+1) \stackrel{[ f_x f_y f_z]}{\longrightarrow} \Jac (f) \to 0,  
\]
where $\Syz(\Jac(f))$ denotes the syzygy module of $\Jac (f)$. It is a reflexive $R$-module of rank two. The line arrangement $\mathcal A (f)$ is said to be \emph{free} if $\Syz(\Jac(f))$ is a free $R$-module. Since we are in $\mathbb P^2$, this is the case if and only if the Jacobian ideal $\Jac (f)$ is saturated. Terao conjectured in 1980 (cf.\ \cite[Conjecture 4.138]{OT})  that freeness of line arrangements (even of hyperplane arrangements)
is a combinatorial property, that is, it depends only on the incidence lattice of the lines (or hyperplanes) in $\mathcal A (f)$. This conjecture is still open and has motivated much work in the theory of line arrangements. 

It is useful to recall that the syzygy module of $\Jac (f)$ admits another description. Define the submodule 
$D(f)\subset R\frac{\partial}{\partial x}\oplus R\frac{\partial}{\partial y}\oplus R\frac{\partial}{\partial z}\cong R^3$ 
to be the set of  $K$-linear derivations $\delta$ such that $\delta(f)\in Rf$.
In particular,  $D(f)$ contains the Euler derivation 
$\delta_E =x\frac{\partial}{\partial x}+y\frac{\partial}{\partial y}+z\frac{\partial}{\partial z}$.  The quotient  module $D_0(f)=D(f)/R\delta_E$ is called the module of derivations of $\mathcal A (f)$. It is isomorphic to $\Syz(\Jac(f))$ up to a degree shift. More precisely, there is an isomorphism of graded $R$-modules 
$D_0 (f)  \cong \Syz (\Jac(f)) (d-1)$ (see, e.g., \cite[Remark A.3]{CHMN}). Since $\Syz (\Jac(f))$ is reflexive, the sheafification $\mathcal D_f$ of $D_0(f)$ is a vector bundle of rank two on $\PP^2$,   called the {\em derivation bundle} of $\mathcal A (f)$. Thus, its restriction to a general  line $L \cong \PP^1$ splits as a direct sum of line bundles  $\mathcal O_{\PP^1}(-a_f) \oplus \mathcal O_{\PP^1}(-b_f)$ for positive integers $a_f, b_f$ satisfying $a_f + b_f = \deg f -1 = d-1$. The ordered pair $(a_f, b_f)$, with $a_f \le b_f$,  is called the {\em splitting type} of $\mathcal A (f)$. 

Projective duality relates any finite set $Z \subset \PP^2$ to a line arrangement. In more detail, any point $[a : b : c] \in \PP^2$ corresponds to a line in the dual projective space defined by $ax + by + cz \in R$. Suppose that $Z$ consists of $d$ points. Using this duality, $Z$ defines a line arrangement $\mathcal A \subset \PP^2$ of $d$ lines that we denote by  $\mathcal A(Z)$.  Similarly, we write $(a_Z, b_Z)$ for  the splitting type of this line arrangement, and so $a_Z \le b_Z$ and $a_Z + b_Z = d-1$. The splitting type is intimately related to the invariants appearing in \Cref{range unexpected curves} that govern the degrees of unexpected curves of $Z$.

\begin{lemma}[{\cite[Lemma 3.5]{CHMN}}]
          \label{lem:invariants}
For any finite set $Z \subset \mathbb P^2$, one has the following relations: 
\begin{itemize}
\item[(a)] $m_Z = a_Z$. 


\item[(b)]  $u_Z = b_Z-1$ (hence $m_Z-1\leq u_Z$).


\item[(c)] $ m_Z \le t_Z \le \big  \lfloor \frac{|Z|-1}{2} \big  \rfloor$.

\item[(d)] If $m_Z<t_Z$, then $t_Z\leq u_Z$.

\end{itemize}
\end{lemma}

Using the connection to line arrangements, one obtains configurations of points that admit unexpected curves. 

\begin{example}[{\cite[Example 6.1]{CHMN}}]
    \label{H19Example}
Consider the line arrangement $\mathcal A$ given by the lines defined by the following 19 linear forms:
$x$, $y$, $z$, $x+y$, 
$x-y$, $2x+y$, $2x-y$, $x+z$, 
$x-z$, $y+z$, $y-z$, $x+2z$, 
$x-2z$, $y+2z$, $y-2z$, $x-y+z$, 
$x-y-z$, $x-y+2z$, $x-y-2z$, shown in Figure \ref{H19Fig}.
Let $Z$ be the corresponding set consisting of the 19 points dual to the lines,
sketched in Figure \ref{dualH19Fig}.

\begin{figure}[htbp]
\centering
\begin{tikzpicture}[line cap=round,line join=round,>=triangle 45,x=.5cm,y=.5cm]
\clip(-5.44,-4.76) rectangle (6.16,4.66);
\draw [domain=-5.44:6.16] plot(\x,{(-0.0-1.0*\x)/-1.0});
\draw (0.0,-4.76) -- (0.0,4.66);
\draw [domain=-5.44:6.16] plot(\x,{(-0.0-0.0*\x)/1.0});
\draw [domain=-5.44:6.16] plot(\x,{-2.0*\x});
\draw [domain=-5.44:6.16] plot(\x,{(-0.0--2.0*\x)/1.0});
\draw [domain=-5.44:6.16] plot(\x,{(-0.0--1.0*\x)/-1.0});
\draw [domain=-5.44:6.16] plot(\x,{(-2.0-0.0*\x)/-1.0});
\draw [domain=-5.44:6.16] plot(\x,{(-1.0-0.0*\x)/-1.0});
\draw [domain=-5.44:6.16] plot(\x,{(--1.0--0.0*\x)/-1.0});
\draw [domain=-5.44:6.16] plot(\x,{(--2.0--0.0*\x)/-1.0});
\draw (-2.0,-4.76) -- (-2.0,4.66);
\draw (-1.0,-4.76) -- (-1.0,4.66);
\draw (1.0,-4.76) -- (1.0,4.66);
\draw (2.0,-4.76) -- (2.0,4.66);
\draw [domain=-5.44:6.16] plot(\x,{(--1.0--1.0*\x)/1.0});
\draw [domain=-5.44:6.16] plot(\x,{(--2.0--1.0*\x)/1.0});
\draw [domain=-5.44:6.16] plot(\x,{(-1.0--1.0*\x)/1.0});
\draw [domain=-5.44:6.16] plot(\x,{(-4.0--2.0*\x)/2.0});
\end{tikzpicture}
\caption{A configuration of 19 lines (the line at infinity, $z=0$, is not shown).}
\label{H19Fig}
\end{figure}
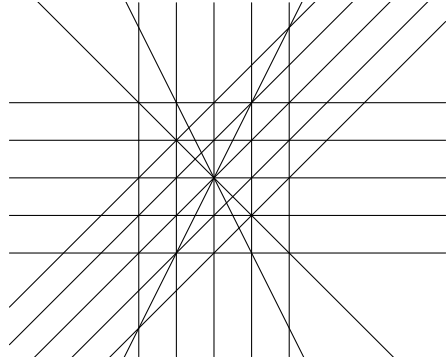

\begin{figure}[htbp]
\centering
\begin{tikzpicture}[line cap=round,line join=round,>=triangle 45,x=1.0cm,y=1.0cm]
\clip(-1.463688897430969,-0.9982420398488032) rectangle (4.481599077455316,3.829741815619195);
\draw (0.0,-0.9982420398488032) -- (0.0,3.829741815619195);
\draw [domain=-1.463688897430969:4.481599077455316] plot(\x,{(--9.0-3.0*\x)/3.0});
\draw [domain=-1.463688897430969:4.481599077455316] plot(\x,{(-0.0-0.0*\x)/3.0});
\draw [domain=-1.463688897430969:4.481599077455316] plot(\x,{(-0.0--1.58*\x)/1.42});
\draw (3.6615593567813454,0.5) node[anchor=north west] {$x$};
\draw (0.5249074252034089,2.845694150810431) node[anchor=north west] {$z$};
\draw (0.05,3.8) node[anchor=north west] {$y$};
\draw (2.1,2.5586802485745417) node[anchor=north west] {$x+y$};
\begin{scriptsize}
\draw [fill=black] (0.0,0.0) circle (1.5pt);
\draw [fill=black] (0.0,3.0) circle (1.5pt);
\draw [fill=black] (3.0,0.0) circle (1.5pt);
\draw [fill=black] (0.0,0.7) circle (1.5pt);
\draw [fill=black] (0.0,1.26) circle (1.5pt);
\draw [fill=black] (0.0,1.76) circle (1.5pt);
\draw [fill=black] (0.0,2.24) circle (1.5pt);
\draw [fill=black] (1.42,1.58) circle (1.5pt);
\draw [fill=black] (0.91,2.09) circle (1.5pt);
\draw [fill=black] (1.86,1.140) circle (1.5pt);
\draw [fill=black] (2.35,0.65) circle (1.5pt);
\draw [fill=black] (0.38753589789044496,0.43120191455415724) circle (1.5pt);
\draw [fill=black] (0.7492696330437865,0.8336943804290022) circle (1.5pt);
\draw [fill=black] (1.0403935472433967,1.157620989186315) circle (1.5pt);
\draw [fill=black] (1.7896631802871834,1.9913153696153176) circle (1.5pt);
\draw [fill=black] (0.54,0.0) circle (1.5pt);
\draw [fill=black] (1.12,0.0) circle (1.5pt);
\draw [fill=black] (1.66,0.0) circle (1.5pt);
\draw [fill=black] (2.22,0.0) circle (1.5pt);
\end{scriptsize}
\end{tikzpicture}
\caption{A sketch of the points dual to the lines of the line configuration given in Figure \ref{H19Fig}.}
\label{dualH19Fig}
\end{figure}
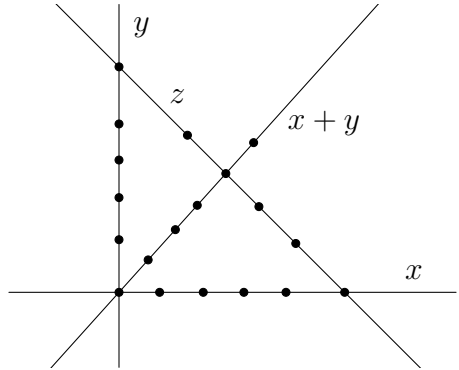

\noindent 
Using a computer algebra system,  one can show that the splitting type of $\mathcal A$ is $(8, 10)$.  Thus, $Z$ admits an unexpected curve only 
in degree 9. This unexpected curve is not irreducible.  Indeed, it has two components, one of 
which is a line  joining the general point to the point $[2 : 1 : 0]$. 

\end{example}

The properties of the set $Z$ are sensitive to slight modifications of $Z$. 

\begin{example}[{\cite[Example 6.1]{CHMN}}]
    \label{example20}
Based on computer experiments, it turns out that  the arrangement of  \Cref{H19Example} is not free, but that if we either (i) remove $2x+y$ alone or (ii) replace $2x+y$ by $2y-x$ or (iii) add $(2y-x)$ to the configuration of 19 lines, these new configurations (see \Cref{H20Fig}) are free with splitting type (respectively) $(7,10)$,  $(7, 11)$ or $(8,11)$.

\begin{figure}[htbp] 
\centering
\begin{tikzpicture}[line cap=round,line join=round,>=triangle 45,x=.5cm,y=.5cm]
\clip(-5.44,-4.76) rectangle (6.16,4.66);
\draw [domain=-5.44:6.16] plot(\x,{(-0.0-1.0*\x)/-1.0});
\draw (0.0,-4.76) -- (0.0,4.66);
\draw [domain=-5.44:6.16] plot(\x,{(-0.0-0.0*\x)/1.0});
\draw [dash pattern=on 5pt off 2pt,domain=-5.44:6.16] plot(\x,{-2.0*\x});
\draw [domain=-5.44:6.16] plot(\x,{(-0.0--2.0*\x)/1.0});
\draw [domain=-5.44:6.16] plot(\x,{(-0.0--1.0*\x)/-1.0});
\draw [domain=-5.44:6.16] plot(\x,{(-2.0-0.0*\x)/-1.0});
\draw [domain=-5.44:6.16] plot(\x,{(-1.0-0.0*\x)/-1.0});
\draw [domain=-5.44:6.16] plot(\x,{(--1.0--0.0*\x)/-1.0});
\draw [domain=-5.44:6.16] plot(\x,{(--2.0--0.0*\x)/-1.0});
\draw [dash pattern=on 2pt off 5pt] (-5.44,-2.72) -- (6.16,3.08);
\draw (-2.0,-4.76) -- (-2.0,4.66);
\draw (-1.0,-4.76) -- (-1.0,4.66);
\draw (1.0,-4.76) -- (1.0,4.66);
\draw (2.0,-4.76) -- (2.0,4.66);
\draw [domain=-5.44:6.16] plot(\x,{(--1.0--1.0*\x)/1.0});
\draw [domain=-5.44:6.16] plot(\x,{(--2.0--1.0*\x)/1.0});
\draw [domain=-5.44:6.16] plot(\x,{(-1.0--1.0*\x)/1.0});
\draw [domain=-5.44:6.16] plot(\x,{(-4.0--2.0*\x)/2.0});
\end{tikzpicture}
\caption{A configuration of 20 lines (the line at infinity is not shown).}
\label{H20Fig}
\end{figure}
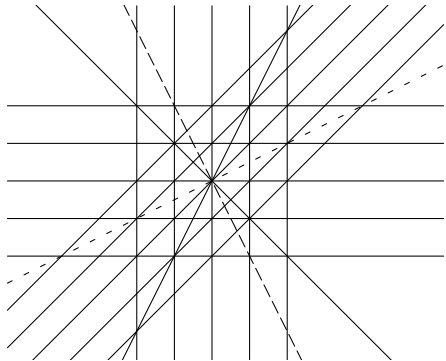

On the dual side, taking $Z$ from Example \ref{H19Example}, if we set $Z_1 = Z \backslash \{ [2,1,0] \}$ and $Z_2 = Z_1 \cup \{ [-1,2,0] \}$ and $Z_3 = Z \cup \{ [-1,2,0] \}$, the mentioned splitting types give $m_{Z_1} = m_{Z_2} = 7$ and $m_Z = m_{Z_3} = 8$. Considering the Hilbert functions, one computes $t_{Z_1} = 8, t_Z = t_{Z_2} = t_{Z_3} = 9$. Hence \Cref{range unexpected curves}  shows that $Z_1, Z_2$ and $Z_3$ have unexpected curves of degree $8, 8$ and $9$, respectively. Moreover,  the unexpected degree 8  curve  for $Z_1$ is irreducible and coincides with the unexpected degree 8 curve for $Z_2$, while the unexpected degree 9 curve for $Z_3$ coincides with that for $Z$ and is not irreducible. As the general point $P$ varies, all unexpected curves for $Z_3$ also contain $[-1,2,0]$.

\end{example}

The connection to line arrangements allows one to establish that a generic finite set of points does not admit any unexpected curve. In fact, there is a more precise result. 

\begin{proposition}[{\cite[Corollary 6.8]{CHMN}}]
      \label{prop:dZ lin gen position} 
If $Z \subset \PP^2$ is a finite set of points such that no three of its points are collinear then $Z$ does not admit any unexpected curve. 
\end{proposition}

The proof utilizes stability properties of derivation bundles and the Grauert-M\"ulich theorem \cite{GM}. It would be interesting to have a more direct proof using just properties of sets of points. 

The existence of unexpected curves is related to the presence of other unexpected properties such as, for example, the failure of a Lefschetz property. This connection was introduced in \cite{DIV} and further elaborated in \cite{DI}, which was motivated by the first version of \cite{CHMN}. Recall that, following \cite{HMNW}, a standard graded Artinian $K$-algebra $A$ is said to have the \emph{Strong Lefschetz Property (SLP)} if there is a linear form $\ell \in [A]_1$ such that each multiplication homomorphism $\times \ell^k \colon [A]_j \to [A]_{j+k}$ with integers $j, k \ge 0$ has maximal rank. There are many classes of Artinian algebras that are expected to have the SLP, see, for example, \cite{MN}. However, deciding the presence of the SLP is often a subtle problem. The following result shows that the failure of the SLP can have an interesting reason. 

\begin{theorem}[{\cite[Theorem 7.5]{CHMN}}] 
    \label{SLP condition}
Let $\mathcal A(f)$ be a line arrangement in $\mathbb P^2$ with $f = \ell_1\cdots \ell_d$, and let $Z \subset \PP^2$ be the  dual set of points.  Then $Z$ has an unexpected curve of degree $j+1$ if and only if, for any general linear form $\ell$, the multiplication homomorphism $\times \ell^2 \colon [A]_{j-1} \to [A]_{j+1}$ does not have maximal rank, where  $A = R/(\ell_1^{j+1},\dots,\ell_d^{j+1})$.
\end{theorem}

The combinatorial properties of a line arrangement are captured by its \emph{incidence lattice}.  It consists of all intersections of lines, ordered by reverse inclusion.  For example, if $\mathcal{A} (f)$ and $\mathcal{A} (g)$  are two line arrangements in $\PP^2$ with the same incidence lattice, then it follows that the Jacobian ideals of $f$ and $g$ have the same degree. It is natural to wonder to what extent numerical invariants of a line arrangement are determined by its combinatorial properties. In this spirit, Terao conjectured that freeness of a line arrangement  depends  only on its incidence lattice.

It turns out that there is connection between Terao's conjecture and the multiplication by the square of a general linear form on certain quotient algebras. This was first studied in \cite{DIV}. Subsequent results in \cite{CHMN} led to the following equivalence. 

\begin{proposition}[{\cite[Proposition 7.13]{CHMN}}] 
   \label{prop:Terao equiv} 
The following two conditions are equivalent: 
\begin{itemize}

\item[(a)] Terao's conjecture is true. 

\item[(b)] If $\mathcal A (f)$ is any free line arrangement with splitting type $(a, b)$, then, for every line arrangement $\mathcal A (g)$ with the same incidence lattice as $\mathcal A (f)$, the multiplication map 
\[
[R/J]_{b-2} \stackrel{\times L^2}{\longrightarrow} [R/J]_b
\]
is surjective, where $J = (\ell_1^b,\dots, \ell_{a+b+1}^b, L_1^b,\ldots,L_{b-a}^b)$ with $g = \ell_1 \cdots \ell_{a+b+1}$ and general linear forms $L, L_1,\ldots,L_{b-a} \in R$. 
\end{itemize}
\end{proposition} 

Combining the arguments for its proof and \Cref{lem:invariants} one is led to a conjectural property of the multiplicity index of a set of points that implies Terao's conjecture. 

\begin{proposition}[{\cite[Corollary 7.15]{CHMN}}]
     \label{prop:points to Terao}
If, for sets of  $2k + 1$ points of $\PP^2$, having (maximal) multiplicity index $k$ is a combinatorial property, then  Terao's conjecture is true for line arrangements.   
\end{proposition}
\smallskip

The definition of an unexpected curve invites extensions. In fact, such extensions have been considered in the literature. The first step is to pass from the projective plane to any projective space. 
Consider a point $P \in \PP^n$ defined by the ideal $I_P \subset R = K[x_0,\ldots,x_n]$. The vector space of degree $d$ forms defining hypersurfaces of $\PP^n$ that have multiplicity at least  $m$ at $P$ is $[I_P^m]_d$. In fact, the fat point scheme $mP$ is defined by the ideal $I_{mP} = I_P^m$
(note that $I_P^m$ is a saturated ideal). One may also consider hypersurfaces that vanish with some multiplicity at more than one general point. For results in this direction, we refer to \cite{Szp1}. 

The second step is to consider also subschemes of positive dimension. 
Let $Z \subset \PP^n$ be any projective subscheme. Thus, the linear system of degree $d$ hypersurfaces that contain $Z$ and have multiplicity at least $m$ at $P$ corresponds to $[I_Z \cap I_P^m]_d$.  Its actual dimension is 
\[
{\adim} (Z,d,m) = \dim [I_Z \cap I_P^m]_d.  
\]
The requirement that a degree $d$ hypersurface has multiplicity at least $m$ at a point $P \in \PP^n$ imposes $\binom{n+m-1}{n}$ linear conditions. Thus, if $P$ is a general point then the expected dimension of the linear system $[I_Z \cap I_P^m]_d$ is 
\[
{\edim} (Z,d,m) = \max \left \{ 0, \dim [I_Z]_d - \binom{n+m-1}{n} \right \}. 
\]
We say that  $Z \subset \PP^n$  {\em admits an unexpected hypersurface} of degree $d$ with a general point $P$ of multiplicity $m$ if the actual dimension of the corresponding linear system is greater than the expected dimension, that is, if 
\begin{equation}
  \label{eq:unexpected hype}
 \dim [I_{Z+mP}]_{d}   > \max \left  \{0, \ \dim [I_Z]_{d} - \binom{n+m-1}{n}\right \}.
\end{equation}
Sometimes, we also  say that $Z \subset \PP^2$  admits an unexpected hypersurface of degree $d$ with respect to $mP$. 

We have seen above that the characterization of sets of points that admit an unexpected hypersurface (of any degree) is challenging even if $n=2$. A characterization of all subschemes that admit some unexpected hypersurface is currently out of reach. 
A more modest question is to find conditions on the triples $(n, d, m)$ such that there is a finite subset $Z \subset \PP^n$ that admits  an unexpected hypersurface of degree $d$ with a general point of multiplicity $m$. It turns out that this question has a very clean and complete answer. 

\begin{theorem}[{\cite[Theorem 1.2]{HMNT}}]
     \label{mainThm}
Given positive integers $(n,d,m)$ with $n \ge 2$, there exists an unexpected hypersurface of degree $d$ with a general point of multiplicity $m$ for some finite subset  
$Z \subset  \PP^n$  if and only if one of the following conditions holds true:
\begin{enumerate}[label={(\roman*)}]
\item[(a)] $n=2$ and $(d,m)$ satisfies $d > m > 2$; or
\item[(b)] $n \ge 3$ and $(d,m)$ satisfies $d \geq m \geq 2$.
\end{enumerate}
\end{theorem} 

If $n = 2$ and $d = m+1$, the existence of a suitable set $Z \subset \PP^2$ is guaranteed for the most part by examples in \cite{CHMN}. The argument in the other cases is constructive and uses cones. By a cone with vertex $P$ we mean a subscheme $X$ such that for every point $Q$ in $X$ the line joining $P$ and $Q$ is in $X$. Observe that, by B\'ezout’s theorem, every hypersurface of degree $d$ with a point $P$ of multiplicity $d$ is a cone with vertex $P$. 

The key for establishing \Cref{mainThm} is the insight that many non-degenerate curves in $\PP^3$ admit an unexpected hypersurface. 

\begin{proposition}[{\cite[Proposition 2.1]{HMNT}}] 
    \label{cone in P3}
Let $C \subset \PP^3$ be a reduced, equidimensional, non-degenerate curve of degree $d$ and let $P \in \PP^3$ be a general point. 
Then the cone over $C$ with vertex $P$ is an unexpected surface of degree $d$ for $C$ with multiplicity $d$ at $P$. It is the unique unexpected surface of $C$ of this degree and multiplicity.
\end{proposition}

Note that the curve $C$ may be reducible, disconnected and/or singular. Its degree $d$ must be at least two since $C$ is non-degenerate. 

Using general hyperplane sections, the above result can be extended to codimension two subschemes of any $\PP^n$ with $n \ge 3$ (see \cite[Proposition 2.4]{HMNT}) and then be used to produce many sets of points with an unexpected hypersurface. 

Another source for examples of finite point sets that can have an unexpected hypersurface are hyperplane arrangements since any hyperplane of $\PP^n$ corresponds to a point in the dual projective space, see for example \cite{CHMN} and \cite{HMNT}. Root systems of various reflection groups give examples of sets of points with an unexpected hypersurface  \cite{HMNT}. However, not all root systems seem to give rise to unexpected hypersurfaces. It would be interesting to understand precisely which root systems do. In a similar direction, the so-called supersolvable line arrangements that give unexpected curves have been characterized in \cite{DMO}. Furthermore, as we will see in the next section, geproci sets give rise to unexpected hypersurfaces. 
\smallskip

A further extension of the concept of an unexpected hypersurface was developed by Trok in \cite{trok}, building on \cite{DHST, DHRST, HMT}.  Instead of considering a general point Trok studies the conditions imposed by a general linear subspace.  The starting point is the following observation. 

\begin{lemma}[{\cite[Page 14]{trok}, \cite[Lemma 2.1]{DHST}}]
   \label{lin space}
For any linear subspace $Q$ of $\PP^n$, one has for any integers $d \ge m-1 \ge 0$, 
\begin{align*}
\dim [R/I_Q^m]_d & = \sum_{i = 0}^{m-1} \binom{\dim Q + d - i}{\dim Q} \binom{\codim Q - 1 + i}{i} \\
& = \binom{n+d}{n} - \sum_{i = m}^{d} \binom{\dim Q + d - i}{\dim Q} \binom{\codim Q - 1 + i}{i}. 
\end{align*}
\end{lemma} 

This computation is based on the fact that $I_Q^i/I_Q^{i+1}$ is a free $R/I_Q$-module generated by 
$[I_Q^i/I_Q^{i+1}]_i$. 

\Cref{lin space} shows that having multiplicity $m$ at $Q$ imposes  $\sum_{i = m}^{d} \binom{\dim Q + d - i}{\dim Q} \binom{\codim Q - 1 + i}{i}$ linear conditions on hypersurfaces of degree $d$. Thus, one says that a projective subscheme $Z \subset \PP^n$  {\em admits an unexpected hypersurface} of degree $d$ with a general linear subspace $Q$ of multiplicity $m$ if 
\begin{align*}
  \label{eq:unexpected mQ}
 \dim [I_{Z+mQ}]_{d} &  = \dim [I_Z \cap I_Q^m]_d   \\
 & > \max \Big  \{0, \ \dim [I_Z]_{d} - \binom{n+d}{n}  + \dim [I_Q^m]_d \Big \} \\
 & = \max \left  \{0, \ \dim [I_Z]_{d} - \binom{n+d}{n} 
 +  \sum_{i = m}^{d} \binom{\dim Q + d - i}{\dim Q} \binom{\codim Q - 1 + i}{i} \right \}.
\end{align*}
If $Z$ is a finite set this is Trok's definition (see \cite[Definition 5.2, Proposition 5.3]{trok}). 
Following Trok, we also refer to such an unexpected hypersurface as an unexpected $mQ$-hypersurface of degree $d$ of $Z$. 

Trok points out that in some cases the existence of an unexpected hypersurface is actually not that surprising. 

\begin{example}
    \label{exa:motivated very unexp}
(i) Let $H \subset \PP^n$ be a proper non-empty linear subspace and $d \ge 2$ be any integer. Then $H$ admits an unexpected degree $d$ hypersurface with a general linear codimension two subspace $Q \subset \PP^n$ of multiplicity $d-1$ if and only if the dimension of $H$ is at least two. 

Indeed, according to \cite[Proposition 5.5]{trok}, one has on the one hand
\begin{equation}
      \label{eq:lin space}
\dim [I_{H + (d-1)Q}]_d = d \cdot \codim H.       
\end{equation}
On the other hand,  \Cref{lin space} gives that the expected dimension of $[I_{H + (d-1)Q}]_d$ is 
\[
\max \Big \{ 0, \dim [I_Q^{d-1}]_d -  \dim [R/I_H]_d   \Big \} = 
\max \Big \{ 0, d n + 1 - \binom{d + \dim H}{d}  \Big\}. 
\]
Comparing these formulas, it follows that the actual dimension of $[I_{H + (d-1)Q}]_d$ is greater than its expected dimension unless 
the dimension of $H$ is zero or one. 

(ii) (\cite[Example 5.6]{trok}) 
Let $H \subset \PP^n$ be a proper linear subspace of dimension $h \ge 1$ and  $Q \subset \PP^n$
be a general linear codimension two subspace. 
Fix  any integer $d \ge 2$.   Let $W$ be a subset of $\binom{h+d}{d}$ general points of $H$. Thus, $[I_W]_d = [I_H]_d$, and so Formula \eqref{eq:lin space} gives 
\[
\dim [I_{W + (d-1)Q}]_d = d \cdot (n - h).
\]
Let $Z$ be a subset of $\PP^n$ obtained from $W$ by adding  $s$ further points of $\PP^n$ to $W$, where 
$s < \min \big \{d (n-h), \binom{h+d}{d}- dh \big \}$. Thus, 
\[
\dim [I_{Z + (d-1)Q}]_d \ge d \cdot (n - h) - s > 0, 
\]
and so $Z$ admits an unexpected  $(d-1)Q$-hypersurface of degree $d$. 
\end{example}

In this example, ``unexpectedness'' for $Z$ is explained by the ``unexpectedness'' for $H$ and the fact that sufficiently many points of $Z$ are in $H$. This motivates the following more restrictive concept of Trok. 

\begin{definition}[{\cite[Definition 5.7]{trok}}] 
   \label{def:very unexpected}
A projective subscheme $Z \subset \PP^n$ is said to   {\em admit a very unexpected hypersurface} of degree $d$ with a general linear subspace $Q \subset \PP^n$ of multiplicity $m$ if there is a subset $W \subseteq Z$ satisfying the  following conditions: 
\begin{itemize}
\item[(i)] $[I_{W + mQ}]_d = [I_{Z + mQ}]_d$. 

\item[(ii)] For any irreducible subvariety $X \subset \PP^n$, one has 
\[
|W \cap X | \le \dim [I_Q/I_{X+mQ}]_d. 
\]

\item[(iii)] 
\begin{align*}
 \dim [I_{W+mQ}]_{d} 
 & > \max \Big  \{0, \ \dim [I_W]_{d} - \binom{n+d}{n}  + \dim [I_Q^m]_d \Big \} \\
 & = \max \left  \{0, \ \dim [I_W]_{d} - \binom{n+d}{n} \right. + \\
& \hspace*{3cm} \left.   \sum_{i = m}^{d} \binom{\dim Q + d - i}{\dim Q} \binom{\codim Q - 1 + i}{i} \right \}.
\end{align*}
\end{itemize} 
\end{definition}

\begin{remark}

(i) It is straightforward to check that the condition 
\[
\dim [R/I_W]_d > \dim [I_Q^m/I_{W + mQ}]_d 
\]
is equivalent to 
\[
 \dim [I_{W+mQ}]_{d}  > \dim [I_W]_{d} - \binom{n+d}{n}  + \dim [I_Q^m]_d. 
\]
Thus, the above definition, while not Trok's original formulation,  agrees with the original \cite[Definition 5.7]{trok}, except for a slight adjustment to account for the requirement $[I_{Z + mQ}]_d \neq 0$. 

(ii) Condition (iii) in \Cref{def:very unexpected} says that $W$ has an unexpected $mQ$-hypersurface of degree $d$, which implies that $Z$ has the same property by Condition (i). 
\end{remark}

The following example illustrates the difference between unexpected and very unexpected hypersurfaces. 

\begin{example}
   \label{rem:unexp vs very unexp} 
Let $H \subset \PP^3$ be any plane and $Q \subset \PP^3$ a general line. Furthermore, let $U$ be a general set of 10 points of $H$. Choose two points $P_1, P_2 \in \PP^3 \setminus H$ and set $Z = U + P_1 +P_2$. We claim that $Z$ has an unexpected $2Q$-hypersurface of degree 3. Indeed, one expects $[I_{Z+2Q}]_3 = 0$ as the expected dimension of  $[I_{Z+2Q}]_3$ is 
\[
\max \Big  \{0, \ \dim [I_Z]_{3} - \binom{3+3}{3}  + \dim [I_Q^2]_3 \Big \} = 0,
\]
where we used the fact that $\dim [I_Z]_3 = 8$ and \Cref{lin space}. 

Consider now the linear forms $\ell, \ell_1, \ell_2$ that define $H$, the span of $Q$ and $P_1$ and the span of $Q$ and $P_2$, respectively. Then $\ell \ell_1 \ell_2$ is in $[I_{Z+2Q}]_3$ and so defines an unexpected $2Q$-hypersurface of $Z$. In fact, it is the unique such unexpected hypersurface as $ \dim [I_{Z+2Q}]_3~=~1$. 

We now show that $Z$ does not have any very unexpected $2Q$-hypersurface of degree 3. Indeed, if there is a subset $W \subset Z$ satisfying the conditions of \Cref{def:very unexpected}, then by using $X= H$ Condition (ii) gives 
\[
|W \cap H | \le \dim [I_Q/I_{H+2Q}]_3 = 7, 
\]
where we used \Cref{lin space} and Formula \eqref{eq:lin space}. It follows that 
$\dim [R/I_W]_3 \le |W|  \le |W \cap H| + 2 \leq 9$. Hence Condition (iii) in  \Cref{def:very unexpected} says
\[
\dim [I_{W+mQ}]_{d}  > \max \Big  \{0, \ \dim [I_W]_{d} - \binom{n+d}{n}  + \dim [I_Q^m]_d \Big \} \ge 1. 
\]
This is a contradiction because Condition (i) in  \Cref{def:very unexpected} requires $ \dim [I_{W+2Q}]_3 =  \dim [I_{Z+2Q}]_3 = 1$.
\end{example}

Trok introduces a new invariant to characterize the existence of very unexpected hypersurfaces. 

\begin{definition}[{\cite[Definition 5.20]{trok}}]
    \label{def:ExC}
For any finite subset $Z \subset \PP^n$ and integer $d \ge 0$, define the \emph{modified expected number of conditions} imposed by $Z$ as 
\[
\Exc (Z, d) = \min \Big \{ s + d \sum_{i=1}^s \dim H_i \; | \; H_1,\ldots,H_s \neq \emptyset  
\text{ subspaces of  $ \PP^n$ with $Z \subseteq \bigcap_{i = 1}^s H_i $} \Big \}.  
\]
\end{definition} 

Using as subspaces the points of $Z$, one obtains $\Exc (Z, d) \le |Z|$ for each $d$. The new invariant allows Trok to establish the following criterion. 

\begin{theorem}[{\cite[Theorem 5.21]{trok}}] 
   \label{thm:very unexp} 
A finite subset $Z$ of $\PP^n$ admits a very unexpected hypersurface of degree $d$  with a general  codimension two linear subspace $Q \subset \PP^n$  of multiplicity $d-1$ if and only if 
\[
\dim [I_{Z + (d-1)Q}]_d > n d + 1 - \Exc (Z, d). 
\]
\end{theorem}
 
Note that the condition stated in \cite[Theorem 5.21]{trok} is equivalent to 
\begin{align*}
\dim [I_{Z + (d-1)Q}]_d & > \dim [I_Q^{d-1}]_d - \Exc (Z, d) \\
& = n d + 1 - \Exc (Z, d), 
\end{align*}
where the equality is a consequence of \Cref{lin space}. The proof of \Cref{thm:very unexp} is based on a duality result \cite[Theorems 4.14 and 5.27]{trok} that relates unexpected hypersurfaces of $Z$ to the derivation module of the hyperplane arrangement dual to $Z$, extending previous work of Faenzi and Valles \cite[Theorem 4.3]{FV} for points in the projective plane to points in $\PP^n$. 

Coming back to our starting point, finite subsets of a projective plane, Trok shows that the above theorem implies the following result. Recall that we use $(a_Z, b_Z)$ to denote the splitting type of the line arrangement defined by $Z \subset \PP^2$. 

\begin{theorem}[{\cite[Theorem 5.1]{trok}}]
    \label{thm:plane trok}
Any finite set $Z \subset \PP^2$ has exactly one of the following two properties: 
\begin{itemize}

\item[(i)] There is a subset of $Z$ consisting of at least $a_Z + 2$ collinear points, and so $Z$ does not admit any unexpected curve of degree $d$ with a general point of $\PP^2$ of multiplicity $d-1$. 

\item[(ii)] The set $Z$ admits an unexpected curve of degree $d$ with a general point of $\PP^2$ of multiplicity $d-1$ whenever $a_Z < d < b_Z$. 
\end{itemize}
\end{theorem} 

Trok attributes this result to \cite{CHMN}. In fact, it is equivalent to \Cref{thm:geomVersion}. Note that in the proof of \cite[Corollary 5.7]{CHMN} it is shown that the existence of $m_Z + 2 = a_Z + 2$ collinear points in $Z$ forces $m_Z = t_Z$. Furthermore, \Cref{lem:invariants} implies $2 m_Z + 2 \le |Z|$. If equality is true then the same lemma shows $m_Z = u_Z$, and so $Z$ does not have any  unexpected curve. Observe also that the  stated ranges for the degrees of unexpected curves in \Cref{thm:geomVersion} and in \Cref{thm:plane trok}  are the same as $m_Z = a_Z$ and $b_Z = |Z| - 1 - m_Z$. 

Let $Z \subset \PP^n$ be a finite subset and $Q \subset \PP^n$ be a general linear subspace of codimension two. As discussed above, if $Z$ has a very unexpected hypersurface of degree $d$ with multiplicity $d-1$ at $Q$, then $Z$ has an unexpected hypersurface of degree $d$ with multiplicity $d-1$ at $Q$. 
In \cite[Remark 5.11]{trok}, Trok points out that in this case the converse is true as well if $n=2$.
Thus, \Cref{thm:very unexp} may be viewed as an extension of \Cref{thm:geomVersion} for $\PP^2$ to any $\PP^n$ with $ n \ge 2$. 

Trok also establishes new constraints on sets in the plane that admit an unexpected curve. For example, he shows: 

\begin{theorem}[{\cite[Theorem 7.6]{trok}}]
If a finite set $Z \subset \PP^2_{\CC}$ admits an unexpected curve of degree $d$ then $|Z| \le 3 d - 3$. 
\end{theorem}


\section{Geproci sets} \label{brian}

Recall that a finite set $Z\subset \PP^n$ has an unexpected hypersurface in degree $d$ with respect to a general point $P\in\PP^n$ of multiplicity $m$ if
\begin{equation}
\dim [I_Z\cap I_P^m]_d >\max\left\{0, \dim [I_Z]_d -\binom{n+m-1}{n}\right\}.
\end{equation}

An interesting instance of this is that of {\it unexpected cones}; i.e., the case where $d=m$ (\cite{HMNT, HMT}).
In some special cases when $n=3$ there are unexpected cones of two degrees, say degrees $a\leq b$, where the cones meet transversely 
along $ab$ concurrent lines. Examples of this for nondegenerate points $Z\subset \PP^n$ were first noticed for point sets coming from
the root systems $D_4$ and $F_4$ \cite{HMNT, CM}, and later confirmed also for $H_4$ \cite{WZ}. 
This means the intersection of the two cones with a plane $H$ gives a complete intersection of two curves of degrees $a\leq b$ in $H$.
This leads to the notion of an $(a,b)$-{\it geproci} set \cite{CFFHMSS, PSS}.

\begin{definition}
A finite set $Z\subset\PP^3$ is
$(a,b)$-geproci 
if its image $\overline Z$, under projection from a general point $P$ to a plane $H$, is an $(a,b)$-complete intersection for some $a\leq b$. If $Z$ is $(a,b)$-geproci for some $a$ and $b$, we simply say $Z$ is geproci.
\end{definition}

In particular, the $D_4$ configuration is a 12 point $(3,4)$-geproci set, the $F_4$ configuration is a 24 point $(4,6)$-geproci set,
and the $H_4$ configuration is a 60 point $(6,10)$-geproci set (see \cite{CM, WZ}). 

\subsection{Visualizing the $D_4$ and $F_4$ configurations}
The 12 point $D_4$ configuration can be regarded as coming from a cube in 3 point perspective.
The 12 points consist of the 8 vertices of the cube, the center of the cube, and the 3 points of perspective.
In Figure \ref{FigD4}, the center of the cube and the rear vertex are not visible. Note this pair of obscured points defines a line 
$L$ containing the vertex represented in the figure as a white dot. The line $L$, together
with the three lines highlighted in gray give 4 skew lines, each of which contains 3 of the 12 points of the $D_4$ configuration.

The cone over these lines with vertex a general point $P$ gives a quartic surface consisting of 4 planes through
$P$ that contains the 12 points of the $D_4$ configuration. Since the $D_4$ configuration is $(3,4)$-geproci
there must be a cubic cone with vertex $P$ containing all 12 points and meeting each of the 4 lines transversely.
In fact, there are two cubic cones containing the 9 points of the $D_4$ configuration not on $L$:
the cone over the 3 gray lines is one and the cone over the three dashed lines is another.
(Thus these 9 points form a $(3,3)$-grid.)
These two cones define a pencil of cubic cones, all with vertex $P$, and it is clear that in this pencil there
is one cone that contains the white dot. One can verify that this cone contains all 12 points and meets $L$ and each of the 
gray lines transversely. This cubic cone, and the quartic cone from above, are the cones with respect to which the $D_4$ configuration
is $(3,4)$-geproci. 

The $F_4$ configuration also has a visualization based on a cube in 3 point perspective. A cube has 8 vertices;
in Figure \ref{FigF4}, these are shown as black dots (the one at the back of the cube is not visible). These 8 vertices give
$28=\binom{8}{2}$ pairs of points, each of which defines a line, giving 28 distinct lines. Of these 28 lines,
12 are edges of the cube (9 of which are visible in the figure), 12 are diagonals on faces of the cube (6 of which are visible in the figure, as dotted lines), 
and 4 are diagonals through the center
of the cube. These 28 lines meet at only 24 points: 12 of these points of intersection are the points of 
the $D_4$ configuration (10 of which are visible in the figure, 7 of which are black dots; the 3 points of perspective,
which are at infinity, are visible as dotted white dots) and 6 are the centers of faces of the cube (three of which are visible in the figure,
as gray dots). Note that the 12 lines giving 
diagonals on the faces of the cube come in 6 parallel pairs. Each pair meets at a point at infinity, giving
6 more points of intersection for a total of 24 (these are the 6 white dots on the gray lines); 
these 24 points are the points of the $F_4$ configuration, which we see
contains the $D_4$ configuration. 

The 24 points of the $F_4$ configuration can be partitioned into 
6 disjoint sets of 4 collinear points. The cone (with vertex a general point $P$)
over the 6 corresponding lines gives a sextic surface containing all 24 points.
To see one instance of this partition, imagine the cube is sitting with on face on the ground.
Pick a diagonal on a vertical face of the cube. Obtain 3 more diagonals by rotating the cube,
in $90^\circ$ increments,
around the vertical line $V$ through the center of the cube. Each of these 4 diagonals contains 4 points
of the $F_4$ configuration. Four of the remaining 8 points lie on $V$ and the last 4 points
lie on the line at infinity through the 2 points not on $V$ of 3 point perspective of the cube.
These 6 lines are skew and each contains 4 of the 24 points of $F_4$.

To construct the quartic cone with vertex $P$ which contains all 24 points of $F_4$ and which
transversely meets each of the 6 lines defining the sextic cone, consider that these 6 lines include
4 of the 8 diagonals on vertical faces of the cube. The cone (with vertex $P$) over these 4 diagonals
and the cone (with vertex $P$) over the other 4 diagonals on vertical faces of the cube give two 
quartic surfaces each of which contains 16 of the 24 points of $F_4$. 
(These 16 points form a $(4,4)$-grid; the two sets of 4 diagonals are the grid lines.)
Pick one of the 8 remaining points; cal it $p$.
These two quartic cones define a pencil of cones, one member of which must contain
$p$. It turns out this cone contains all 8 of the remaining points and meets the 6 lines transversely.
This quartic cone, and the sextic cone from above, are the cones with respect to which the $F_4$ configuration
is $(4,6)$-geproci. 

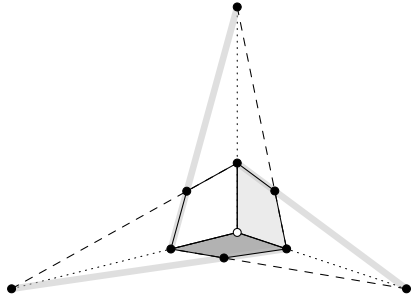
\begin{figure}
\begin{tikzpicture}[line cap=round,line join=round,x=.75cm,y=.75cm]
\clip(-4.496680397937174,0.75) rectangle (5.928312362917013,6.5);
\fill[line width=2.pt,color=black,fill=white] (-1.176350188308613,1.7059124529228467) -- (-0.8953798813051128,2.7315518487365438) -- (0.,3.230935550935551) -- (0.,2.) -- cycle;
\fill[line width=2.pt,color=black,fill=lightgray,fill opacity=0.30000000149011612] (0.,3.230935550935551) -- (0.6656059123409749,2.7359609200174266) -- (0.8751724137931034,1.7082758620689655) -- (0.,2.) -- cycle;
\fill[line width=2.pt,color=black,fill=gray,fill opacity=0.60000000149011612] (-1.176350188308613,1.7059124529228467) -- (-0.23559785204630612,1.546900694008833) -- (0.8751724137931034,1.7082758620689655) -- (0.,2.) -- cycle;
\draw [dotted,line width=0.4pt] (-4.,1.)-- (0.,2.);
\draw [dotted,line width=0.4pt] (0.,2.)-- (0.,6.);
\draw [dotted,line width=0.4pt] (0.,2.)-- (3.,1.);
\draw [line width=2.5pt,color=lightgray,opacity=.5] (-4.,1.)-- (0.8751724137931034,1.7082758620689655);
\draw [line width=2.5pt,color=lightgray,opacity=.5] (0.,3.230935550935551)-- (3.,1.);
\draw [line width=2.5pt,color=lightgray,opacity=.5] (-1.176350188308613,1.7059124529228467)-- (0.,6.);
\draw [line width=0.4pt,color=black] (-1.176350188308613,1.7059124529228467)-- (-0.8953798813051128,2.7315518487365438);
\draw [line width=0.4pt,color=black] (-0.8953798813051128,2.7315518487365438)-- (0.,3.230935550935551);
\draw [line width=0.4pt,color=black] (0.,3.230935550935551)-- (0.,2.);
\draw [line width=0.4pt,color=black] (0.,2.)-- (-1.176350188308613,1.7059124529228467);
\draw [line width=0.4pt,color=black] (0.,3.230935550935551)-- (0.6656059123409749,2.7359609200174266);
\draw [line width=0.4pt,color=black] (0.6656059123409749,2.7359609200174266)-- (0.8751724137931034,1.7082758620689655);
\draw [line width=0.4pt,color=black] (0.8751724137931034,1.7082758620689655)-- (0.,2.);
\draw [line width=0.4pt,color=black] (0.,2.)-- (0.,3.230935550935551);
\draw [line width=0.4pt,color=black] (-1.176350188308613,1.7059124529228467)-- (-0.23559785204630612,1.546900694008833);
\draw [line width=0.4pt,color=black] (-0.23559785204630612,1.546900694008833)-- (0.8751724137931034,1.7082758620689655);
\draw [line width=0.4pt,color=black] (0.8751724137931034,1.7082758620689655)-- (0.,2.);
\draw [line width=0.4pt,color=black] (0.,2.)-- (-1.176350188308613,1.7059124529228467);
\begin{scriptsize}
\draw [fill=black] (-4.,1.) circle (1.5pt);
\draw [fill=white] (0.,2.) circle (1.5pt);
\draw [fill=black] (3.,1.) circle (1.5pt);
\draw [fill=black] (0.,6.) circle (1.5pt);
\draw [fill=black] (-1.176350188308613,1.7059124529228467) circle (1.5pt);
\draw [fill=black] (0.8751724137931034,1.7082758620689655) circle (1.5pt);
\draw [fill=black] (-0.23559785204630612,1.546900694008833) circle (1.5pt);
\draw [fill=black] (0.,3.230935550935551) circle (1.5pt);
\draw [fill=black] (-0.8953798813051128,2.7315518487365438) circle (1.5pt);
\draw [fill=black] (0.6656059123409749,2.7359609200174266) circle (1.5pt);
\end{scriptsize}
\draw [line width=0.4pt,dashed] (3.,1.)-- (-1.176350188308613,1.7059124529228467);
\draw [line width=0.4pt,dashed] (-4.,1.)-- (0.,3.230935550935551);
\draw [line width=0.4pt,dashed] (0.,6.)-- (0.8751724137931034,1.7082758620689655);
\end{tikzpicture}
\caption{The $D_4$ configuration regarded as coming from a cube in 3 point perspective.}\label{FigD4}
\end{figure}

\begin{figure}
\definecolor{zzttqq}{rgb}{0,0,0}
\definecolor{uuuuuu}{rgb}{0.26666666666666666,0.26666666666666666,0.26666666666666666}
\definecolor{ffqqqq}{rgb}{1.,0.,0.}
\begin{tikzpicture}[line cap=round,line join=round,x=1.0cm,y=1.0cm]
\clip(-3.9394825089949665,0.1406056699447517) rectangle (6.485510251859221,8.01327261693464);
\fill[line width=2.pt,color=zzttqq,fill=white,fill opacity=0.5] (-1.3781513347782868,1.8839289642075294) -- (-0.7462561422331093,3.223845547475148) -- (-0.4266366338291549,3.487779668718635) -- (-0.7760364298392128,2.153707720316595) -- cycle;
\fill[line width=2.pt,color=zzttqq,fill=lightgray,fill opacity=0.3] (-0.4266366338291549,3.487779668718635) -- (0.06769653108008261,2.9596182039831778) -- (0.28206782853952683,1.7116230644186283) -- (-0.7760364298392128,2.153707720316595) -- cycle;
\fill[line width=2.pt,color=zzttqq,fill=gray,fill opacity=0.6] (-1.3781513347782868,1.8839289642075294) -- (-0.2623839081803193,1.5980502249179955) -- (0.28206782853952683,1.7116230644186283) -- (-0.7760364298392128,2.153707720316595) -- cycle;
\draw [line width=0.2pt] (-3.5440517491004977,0.9134930642839415)-- (-0.7760364298392128,2.153707720316595);
\draw [line width=0.2pt] (-0.7760364298392128,2.153707720316595)-- (-0.1828902899975091,4.418447526984919);
\draw [line width=0.2pt] (-0.7760364298392128,2.153707720316595)-- (1.8481858858240823,1.0572860678819302);
\draw [line width=0.2pt] (-3.5440517491004977,0.9134930642839415)-- (0.28206782853952683,1.7116230644186283);
\draw [line width=0.2pt] (1.8481858858240823,1.0572860678819302)-- (-1.3781513347782868,1.8839289642075294);
\draw [line width=0.2pt] (-3.5440517491004977,0.9134930642839415)-- (-0.4266366338291549,3.487779668718635);
\draw [line width=0.2pt] (-0.4266366338291549,3.487779668718635)-- (1.8481858858240823,1.0572860678819302);
\draw [line width=0.2pt] (-1.3781513347782868,1.8839289642075294)-- (-0.1828902899975091,4.418447526984919);
\draw [line width=0.2pt] (-0.1828902899975091,4.418447526984919)-- (0.28206782853952683,1.7116230644186283);
\draw [line width=1.pt,color=zzttqq] (-1.3781513347782868,1.8839289642075294)-- (-0.7462561422331093,3.223845547475148);
\draw [line width=1.pt,color=zzttqq] (-0.7462561422331093,3.223845547475148)-- (-0.4266366338291549,3.487779668718635);
\draw [line width=1.pt,color=zzttqq] (-0.4266366338291549,3.487779668718635)-- (-0.7760364298392128,2.153707720316595);
\draw [line width=1.pt,color=zzttqq] (-0.7760364298392128,2.153707720316595)-- (-1.3781513347782868,1.8839289642075294);
\draw [line width=1.pt,color=zzttqq] (-0.4266366338291549,3.487779668718635)-- (0.06769653108008261,2.9596182039831778);
\draw [line width=1.pt,color=zzttqq] (0.06769653108008261,2.9596182039831778)-- (0.28206782853952683,1.7116230644186283);
\draw [line width=1.pt,color=zzttqq] (0.28206782853952683,1.7116230644186283)-- (-0.7760364298392128,2.153707720316595);
\draw [line width=1.pt,color=zzttqq] (-0.7760364298392128,2.153707720316595)-- (-0.4266366338291549,3.487779668718635);
\draw [line width=1.pt,color=zzttqq] (-1.3781513347782868,1.8839289642075294)-- (-0.2623839081803193,1.5980502249179955);
\draw [line width=1.pt,color=zzttqq] (-0.2623839081803193,1.5980502249179955)-- (0.28206782853952683,1.7116230644186283);
\draw [line width=1.pt,color=zzttqq] (0.28206782853952683,1.7116230644186283)-- (-0.7760364298392128,2.153707720316595);
\draw [line width=1.pt,color=zzttqq] (-0.7760364298392128,2.153707720316595)-- (-1.3781513347782868,1.8839289642075294);
\draw [line width=2pt,color=lightgray] plot(\x,{(--15.49213767488207--3.5049544627009777*\x)/3.3611614591029886});
\draw [line width=2pt,color=lightgray] plot(\x,{(--8.359479712193101-3.361161459102989*\x)/2.031076175821591});
\draw [line width=2pt,color=lightgray] (-4,0.9134930642839415) -- (6.4,1.18);
\draw [line width=1.pt,dotted] (-0.7760364298392128,2.153707720316595)-- (-0.7288555655933644,3.8491254481550174);
\draw [line width=1.pt,dotted] (-1.3781513347782868,1.8839289642075294)-- (0.6257878555114629,5.261721528986253);
\draw [line width=1.pt,dotted] (-0.7760364298392128,2.153707720316595)-- (0.4677446196935667,3.341733118912077);
\draw [line width=1.pt,dotted] (0.28206782853952683,1.7116230644186283)-- (-1.9936296081347535,7.414980734875758);
\draw [line width=1.pt,dotted] (-0.7760364298392128,2.153707720316595)-- (0.27625059723465695,1.0153677935195455);
\draw [line width=1.pt,dotted] (-1.3781513347782868,1.8839289642075294)-- (5.618143550697727,1.1578182722785608);
\begin{scriptsize}
\draw [fill=white] (-3.5440517491004977,0.9134930642839415) circle (2.5pt);
\draw [fill=black] (-3.5440517491004977,0.9134930642839415) circle (.5pt);
\draw [fill=black] (-0.7760364298392128,2.153707720316595) circle (2.5pt);
\draw [fill=white] (1.8481858858240823,1.0572860678819302) circle (2.5pt);
\draw [fill=black] (1.8481858858240823,1.0572860678819302) circle (.5pt);
\draw [fill=white] (-0.1828902899975091,4.418447526984919) circle (2.5pt);
\draw [fill=black] (-0.1828902899975091,4.418447526984919) circle (.5pt);
\draw [fill=black] (-1.3781513347782868,1.8839289642075294) circle (2.5pt);
\draw [fill=black] (0.28206782853952683,1.7116230644186283) circle (2.5pt);
\draw [fill=uuuuuu] (-0.2623839081803193,1.5980502249179955) circle (2.5pt);
\draw [fill=black] (-0.4266366338291549,3.487779668718635) circle (2.5pt);
\draw [fill=uuuuuu] (-0.7462561422331093,3.223845547475148) circle (2.5pt);
\draw [fill=uuuuuu] (0.06769653108008261,2.9596182039831778) circle (2.5pt);
\draw [fill=white] (-0.7288555655933644,3.8491254481550174) circle (2.0pt);
\draw [fill=white] (0.6257878555114629,5.261721528986253) circle (2.0pt);
\draw [fill=white] (0.27625059723465695,1.0153677935195455) circle (2.0pt);
\draw [fill=white] (5.618143550697727,1.1578182722785608) circle (2.0pt);
\draw [fill=white] (0.4677446196935667,3.341733118912077) circle (2.0pt);
\draw [fill=white] (-1.9936296081347535,7.414980734875758) circle (2.0pt);
\draw [fill=lightgray] (-0.7542800291228681,2.935511700766003) circle (2.5pt);
\draw [fill=lightgray] (-0.13763731444147465,2.7634889706566796) circle (2.5pt);
\draw [fill=lightgray] (-0.43629001624968966,1.786177856672359) circle (2.5pt);
\end{scriptsize}
\end{tikzpicture}
\caption{The $F_4$ configuration regarded as coming from a cube in 3 point perspective.}\label{FigF4}
\end{figure}
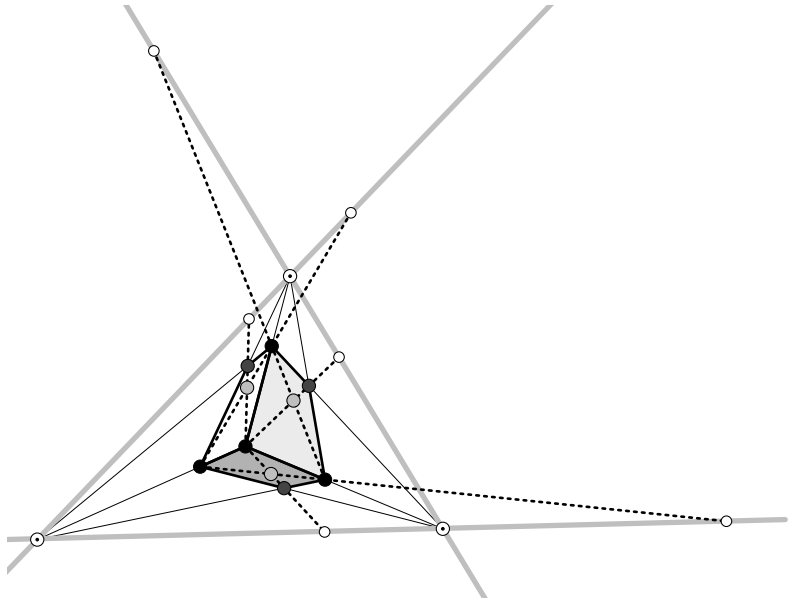

\subsection{The four-fold way}\label{4foldWay}
Geproci sets arise in four ways. First are degenerate sets of points; i.e., a finite set of points 
contained in a plane where the set is already a complete intersection in that plane.
Second are $(a,b)$-{\it grids}; i.e., nondegenerate finite sets of points in $\PP^3$ which are the 
intersection of space curves $A$ and $B$ where $A$ consists of $a$ skew lines and $B$ 
consists of $b$ skew lines, and each component of $A$ meets every component of $B$.
Third are {\it half grids} \cite{PSS}: a finite nondegenerate nongrid set $Z$ is a half grid if it is a geproci set 
whose image $\overline Z$ to a plane under projection from a general point is a complete intersection of two curves, exactly one of which can always be taken to be a union of lines.
(This definition differs slightly from that of \cite{PSS}.) Fourth are the finite nondegenerate nongrid and non-half grid geproci sets
(i.e., this category consists of all geproci sets not in the first three categories). 

Since degenerate geproci sets and grids are well understood, we group them together and call them {\it trivial}.
The {\it nontrivial} geproci sets are the half grids and the nondegenerate non-grid non-half grids.

To see that grids are indeed well understood, note that
the $(2,b)$-grids are exactly the sets of $b$ points on each of two skew lines.
The $(a,b)$-grids with $3\leq a\leq b$ always lie on some smooth quadric surface $Q$,
and are given by taking $A$ to be any set of $a$ lines in one ruling on $Q$ and 
$B$ to be any set of $b$ lines in the other ruling on $Q$.

The nontrivial non-half grids are much less well understood.
The $H_4$ configuration is one of only three currently known examples of a nontrivial geproci non-half grid \cite{WZ, CFFHMSS}.

The half grids are somewhat better understood, at least in the sense that we have many more examples of them
(see Theorem \ref{geographyThm}), beginning with the $D_4$ and $F_4$ configurations. 
The $D_4$ and $F_4$ configurations are half grids
because (as we saw above) one of the cones
with respect to which they are geproci can be taken to be a union of planes.
(In fact, each configuration can be partitioned into collinear subsets defining
skew lines; we saw above that $D_4$ consist of 3 points on each of 4 skew lines, while
$F_4$ consists of 4 points on each of 6 skew lines. It is not obvious that every case of $(a,b)$-geproci half grids 
behaves this way, but not yet published follow up work of the authors of \cite{CFFHMSS} shows that they always consist of either $a$ points on each of $b$ skew lines or
$b$ points on each of $a$ skew lines.)
The other cone with respect to which the $D_4$ and $F_4$ configurations are geproci
(namely the cubic cone for $D_4$ and the quartic cone for $F_4$)
is unique and not a union of planes, and so the $D_4$ and $F_4$ configurations
are half grids and not grids.

\subsection{History}
The question of whether nontrivial geproci sets exist goes back only to 2011 \cite{Polizzi}
(see also \cite[Question 15]{CDFPGR}). At that time the only examples known were what we now call trivial; i.e., either
degenerate or grids. Some unrecognized nontrivial examples were studied in the paper \cite{HMNT};
their significance as nontrivial geproci sets was first noticed during a workshop at a conference
on Lefschetz properties at Levico Terme in 2018 (see the appendix to \cite{CM}).
Additional nontrivial examples were announced at a workshop in 2020 at Oberwolfach \cite{PSS},
which led to the monograph \cite{CFFHMSS}. 

\subsection{Main results}
The paper \cite{CFFHMSS}, by extending results of \cite{CM}, obtains the following classification of nontrivial $(a,b)$-geproci sets 
over the complex numbers for small $a$:

\begin{theorem}[\cite{CFFHMSS}]
Let $Z$ be a nontrivial $(a,b)$-geproci set with $a\leq3$ over the complex numbers. 
Then $|Z|=12$ and $Z$ is projectively equivalent to the $D_4$ configuration.
\end{theorem}

In particular, every nondegenerate
$(2,b)$-geproci set is a grid 
(and hence trivial and contained in a smooth quadric), 
and every nondegenerate $(3,b)$-geproci
set is either a grid (and hence trivial and contained in a smooth quadric)
or it is projectively equivalent to the $D_4$ configuration (and so not contained in a quadric).

Recent work \cite{CFFHMSS2} extends this classification to the case of $(4,4)$-geproci sets,
showing that there are only two nontrivial $(4,4)$-geproci sets, both already known.

By an explicit construction motivated by $D_4$ and $F_4$, \cite{CFFHMSS} also obtains the following result over the complex numbers:

\begin{theorem}[\cite{CFFHMSS}]\label{geographyThm}
For each $4\leq a\leq b$, there is an $(a,b)$-geproci half grid.
\end{theorem}

However, only three examples of nontrivial non-half grids are currently known over the complex numbers: 
the $H_4$ configuration \cite{WZ}, a Gorenstein set of 40 points \cite{CFFHMSS} and a set of 120 points \cite{CFFHMSS}.

\subsection{Generalizations}
One could consider extending the definition of geproci sets to varieties $Z\subset \PP^n$ whose image in a hyperplane under projection from
a general point $P$ is a complete intersection, but no nondegenerate examples are known when $Z$ has codimension greater than 3 for $n>3$,
and only obvious examples coming from taking cones over geproci sets in $\PP^3$ are known when $Z$ has codimension 3.

The geproci property however can be regarded as a special case of a more general notion
which carries over to mathematics the idea of studying structures by inverse scattering.
In an inverse scattering problem one tries to infer the structure of a 3 dimensional object from scattering data.
For example, in tomography, one in essence uses x-rays to project, from various points
in space, a 3 dimensional object as a planar picture, and then aims to recover the internal structure of the original object.
For the geproci property we specify the property of the projection (viz., it is a complete intersection) and we aim
to determine the possible structures of the original point sets. More generally, we could study any specific property
$\mathcal P$ and try to classify what point sets project (under projection from a general point) to sets with property $\mathcal P$
(i.e., we could try to classify the gepro-$\mathcal P$ sets).

\subsection{Open Questions.}

We say a finite set $Z\subset\PP^3$ satisfies $C(t)$ if $Z$ has an unexpected surface of degree $t$ with respect to a general point $t$ of multiplicity $t$; i.e., if it has an unexpected cone of degree $t$.
 
By \cite[Corollary 7.15]{CFFHMSS}, $(2, b)$-grids with $b \geq 2$ do not satisfy $C(b)$, but
by \cite[Theorem 3.5]{CM}, if $Z$ is an $(a,b)$-grid with $a = 2 < b$, then $Z$ satisfies $C(a)$, and $Z$ satisfies $C(b)$ if $2 < a \leq b$. 
Moreover, no nontrivial $(a,b)$-geproci set is known which fails to satisfy both $C(a)$ and $C(b)$.

\begin{question}
Let $Z$ be a nontrivial $(a,b)$-geproci set. Must $Z$ satisfy both $C(a)$ and $C(b)$?
\end{question}

Another open problem involves linear general position (or LGP). A finite set of points $Z\subset\PP^n$ 
is said to be in {\it linear general position} if every subset of $r\leq n+1$ points spans a linear space of dimension
$r-1$. Thus a finite set $Z\subset\PP^3$ in linear general position can have no 3 (or more) points on a line,
and no 4 (or more) points in a plane.

The only known nondegenerate geproci set in linear general position are the $(2,2)$-grids; i.e., sets of 4 points which span $\PP^3$.
All other known examples of nondegenerate geproci sets have a subset of at least 3 collinear points.
This raises the following open question:

\begin{question}\label{LGPques}
Let $Z$ be a nontrivial $(a,b)$-geproci set. 
\begin{enumerate}
\item[(a)] Can $Z$ be in linear general position?
\item[(b)] Must $Z$ have some subset of 3 collinear points?
\end{enumerate}
\end{question}

An easy example of a special set of points in the plane which is not in LGP is 6 or more points on
an irreducible conic. Grids are always on smooth quadrics; 
perhaps quadrics are a place to look for geproci sets which are not in LGP.

\begin{question}
If $Z$ is a nondegenerate geproci set contained in a smooth quadric, must $Z$ be a grid?
\end{question}

A 40 point $(5,8)$-geproci set stands out as currently the only known 
example of a nontrivial geproci set which is also arithmetically Gorenstein \cite[Problem 12]{CFFHMSS}.

\begin{question}
What other nontrivial geproci sets are arithmetically Gorenstein (if any)?
\end{question}

Moduli type questions are also of interest. Using cross-ratios, for any $3\leq a\leq b$ it is easy to see that there are uncountably many
$(a,b)$-grids, no two of which are projectively equivalent. 
Recent work by the authors of \cite{CFFHMSS} shows that there are continuous families
of $(a,mb)$-geproci half grids contained in $a$ skew lines with $mb$ points per line such that
the family contains infinitely many pairwise non-projectively equivalent $(a,mb)$-geproci sets,
whenever $m\geq2$, $b\geq3$, $4\leq a\leq b+1$. These families are based on there being
continuous families of projectively equivalent geproci sets of
$ab$ points with $b$ points on each of the $a$ lines, and building new families
by taking unions of $m$ copies of these sets on the same $a$ lines (in analogy with there being 
continuous families of $(a,b)$-grids coming from unions of $b$ copies of $(1,a)$-grids
with each copy being contained in the same $a$ lines with 1 point per line).
However, for each $3\leq a\leq b$ only finitely many 
nontrivial $(a,b)$-geproci sets are known (and none if $a=3$ \cite{CFFHMSS}) 
up to projectively equivalence
for sets which are not unions of smaller geproci halfgrids on $a$ skew lines.

\begin{question}
Can there be infinitely many nontrivial non-projectively equivalent $(a,b)$-geproci sets 
which are not unions of smaller geproci halfgrids on $a$ skew lines?
\end{question}

Related to this and to Question \ref{LGPques}, one can ask to what extent a nontrivial geproci set is determined by its combinatorics,
where by the {\it combinatorics of a finite set} $Z$ we mean the matroid of linear dependencies among its points, or equivalently
the function which assigns to each subset $S\subset Z$ the dimension of the linear space spanned by $S$. 
Thus two finite sets $Z_1, Z_2\subset \PP^3$ have the same combinatorics if there is a bijection $f:Z_1\to Z_2$ such that
for each subset $S\subset Z_1$ the dimension of the span of $S$ is equal to the dimension of the span of $f(S)$.
Also, by the {\it weak combinatorics of a finite set} $Z$, we will mean the incidence matrix of the points of $Z$ with
the lines through pairs of points of $Z$.

\begin{question}
Let $Z$ be a nontrivial $(a,b)$-geproci set which is
not a union of smaller geproci halfgrids on $a$ skew lines.
\begin{enumerate}
\item[(a)] If $Z'$ is a set with the same combinatorics as $Z$, must
$Z'$ also be a nontrivial $(a,b)$-geproci set, or must $Z'$ even be projectively equivalent to $Z$? 
\item[(b)] If $Z'$ is a set with the same weak combinatorics as $Z$, must
$Z'$ also be a nontrivial $(a,b)$-geproci set, or must $Z'$ even be projectively equivalent to $Z$? 
\item[(c)] If $Z$ and $Z'$ are geproci with the same combinatorics (or weak combinatorics),
must they be projectively equivalent?
\end{enumerate}
\end{question}

More generally, we can ask about gepro-$\mathcal P$ sets for interesting $\mathcal P$.

\begin{question}
Which finite point sets in $\PP^n$ project from a general point to an arithmetically Gorenstein set in a hyperplane? I.e.,
which sets are gepro-AG (arithmetically Gorenstein under a general projection)?
\end{question}

Initial examples of gepro-AG sets are given by any $n+1$ general points in $\PP^n$, but what others are there?


\section{AV sequences - general facts} \label{juan}

In Section \ref{uwe} we have defined what it means for a subscheme of $\PP^n$ to admit an unexpected hypersurface of specified degree and multiplicity at a general point. In effect, there is an expected dimension for a certain linear system but the actual dimension is strictly greater than this expected dimension. We now introduce a third measure, the virtual dimension.

\begin{definition}
Let $X \subset \mathbb P^n$ be a subscheme and let $t \geq m$ be positive integers. Let $P \in \mathbb P^n$ be a general point. The {\it virtual dimension} for $X$ in degree $j$ and multiplicity $m$ is
\[
\hbox{vdim}(X,j,m) = \hbox{vdim} (I_X,j,m) = \dim [I_X]_j - \binom{m+n-1}{n}.
\]
\end{definition}

\begin{remark}
    An equivalent condition for the existence of an unexpected hypersurface of degree $j$ and multiplicity $m$ is that $\hbox{adim} (X,j,m) > 0$ and $\hbox{adim}(X,j,m) > \hbox{vdim}(X,j,m)$.
\end{remark}

We have also defined the virtual dimension as being equivalent to the expected dimension, but differing only in that the virtual dimension is allowed to be negative.

\begin{question} \label{AV question}
Let $X \subset \PP^n$  be a closed subscheme.
\begin{enumerate}
    \item Is there any way to measure how the unexpectedness imposed by $X$ persists as the degree or the multiplicity changes?
    
    \item At the other extreme, are there conditions on $X$ that imply that $X$ does not admit unexpected hypersurfaces (either entirely or in specific ranges)?
\end{enumerate}

\end{question}

In this section we introduce the AV-sequences to give partial answers to these questions. The definition at first may seem a bit strange, but the wealth of results and conjectures that have already arisen  more than justifies this approach.

\begin{definition}[{\cite[Definition 2.8]{FGHM}}]
Let $X \subset \PP^n$ be a closed subscheme. Fixing a non-negative integer $j$, we define the sequence $AV_{X,j}$ as follows:
\[
AV_{X,j}(m) = \hbox{adim}(X,m+j,m) - \hbox{vdim}(X,m+j,m), m \geq 1.
\]
\end{definition}

\begin{remark}[{\cite[Remark 2.9]{FGHM}}]
Notice that if $\hbox{adim}(X,t,m) > 0$ then $X$ admits an unexpected hypersurface of degree $t$ and multiplicity $m$ if and only if $AV_{X,j}(m) > 0$ for $j = t-m$.
\end{remark}

We have the following result for cones, as our starting point.

\begin{proposition}[{\cite[Corollary 2.12]{HMNT}}]
Let $C$ be a reduced, equidimensional, non-degenerate curve of degree $d \geq 2$ in $\PP^3$ ($C$ may be reducible, singular, and/or disconnected). 
Let $P \in \PP^3$ be a general point. 
Let $t \geq d$ be a positive integer. Then $C$ admits an unexpected cone of degree $t$ with multiplicity $t$ at $P$. In particular, $AV_{C,0} (m) > 0$ for $m = t \geq d$. 
\end{proposition}

Still, it is not at all obvious that the AV-sequence (as a function of $m$) should have any nice properties. The first hint is given, though by the following two easy results.

\begin{lemma} [{\cite[Lemma 2.11]{FGHM}}]  \label{FGHM 2.11}
Let $P \in \PP^n$ be a general point and let $X \subset \PP^n$. Then
\[
AV_{X,j}(m) = \dim\left[  \faktor{R}{(I_X + I_P^m)}\right]_{m+j}  .
\] 
\end{lemma}

\begin{proof} 

Set $t:=m+j$. From the short exact sequence
\[
0\to R/(I_X \cap I_P^m) \to R/I_X \oplus R/ I_P^m \to R/(I_X + I_P^m)\to 0
\]
we get  in degree $t$
\[
\dim[R]_t- {\adim} (X,t,m) -\dim[R]_t+\dim[I_X]_t - \dim[R]_t+ \dim [I_P^m]_t +\dim[ R/(I_X + I_P^m)]_t= 0.
\]
Using the fact that $t \geq m$, we get after a calculation
	\[\dim[ R/(I_X + I_P^m)]_t = {\adim} (X,t,m) - {\vdim} (X,t,m) \] 
	as desired.
\end{proof}

\begin{lemma}[{\cite[Remark 2.13]{FGHM}}]  \label{FGHM 2.13}
Denote by $mP$ the zero-dimensional scheme defined by the saturated ideal $I_P^m$. 
If $X$ is ACM of dimension $\geq 1$ then
\[
AV_{X,j}(m) = h^1(\mathcal I_{X + mP} (m+j)).
\]
\end{lemma}

\begin{proof}
This follows from the exact sequence of sheaves
\[
0 \rightarrow \mathcal I_{X + mP} \rightarrow \mathcal I_X \rightarrow \mathcal O_{mP} \rightarrow 0.
\]
We omit the details.
\end{proof}

Of course these last two lemmas show that the integers $AV_{X,j}(m)$ have significant algebraic and cohomological meaning, but they fail to give insight into the {\it sequence} $AV_{X,j}(m)$ as a sequence indexed by $m$. Indeed, Lemma \ref{FGHM 2.11} suggests that the sequence represents the Hilbert function in different degrees of coordinate algebras of varying subschemes, and Lemma \ref{FGHM 2.13} suggests that the sequence represents the dimension of the first cohomology of the ideal sheaf of varying schemes. What is missing is some uniformity to understand the sequence, and this is provided in a surprising way by considering generic initial ideals. Our description comes from \cite{FGHM}.

Let $X \subset \PP^n$ be a closed subscheme with saturated ideal $I_X$. Let $R = K[x_0,x_1,\dots,x_n]$ be a standard graded polynomial ring. Assume that the monomials of $R$ are ordered by $>_{lex}$, the lexicographic monomial order which satisfies $x_0 > x_1 > \dots > x_n$. A set $M \subset R$ of monomials is a {\it lex segment} if the monomials have the same degree and they satisfy the condition that whenever $u,v$ are monomials with $u \geq v$ and $v \in M$ then also $u \in M$. A homogeneous ideal $I \subset R$ is called a {\it lex-segment ideal} if, for each $d$, the homogeneous component $[I]_d$ of $I$ in degree $d$ is generated by a lex segment. 

For a graded ideal $I \subset R$ we denote by $\hbox{gin}(I)$ the generic initial ideal of $I$ with respect to the monomial order $>_{lex}$. A good reference for generic initial ideals is \cite{Green}. 

We collect without proof a technical result from \cite{FGHM}.

\begin{lemma}[{\cite[Lemma 3.1]{FGHM}}] 
Let $X \subseteq \PP^n$ be a closed subscheme. For any non-negative integers $t$ and $m$ we have

\begin{itemize}
    \item[(i)] $\hbox{\rm adim} (X,t,m) = \dim [\hbox{\rm gin} (I_X) \cap I_Q^m ]_t$ where $Q = [1,0,\dots,0] \in \PP^n$.
    
    \item[(ii)] $\hbox{\rm vdim} (X,t,m) = \hbox{\rm vdim} (\hbox{\rm gin} (I_X) , t,m)$. 
\end{itemize}
    
\end{lemma}

Using this result and a fair amount of computation, the following surprising result was then shown.

\begin{theorem}[{\cite[Theorem 3.4]{FGHM}}]  \label{shifted O-seq}
For any non-negative integer $j$, the sequence $AV_{X,j}$ shifted to the left by 1 is an $O$-sequence. In particular, setting $J = \hbox{\rm gin}(I_X) : x_0^{j+1}$, the sequence $AV_{X,j}$ shifted to the left by 1 coincides with the Hilbert function of $R/J$, i.e.
\[
AV_{X,j} (s+1) = h_{R/J}(s), \ \ s \geq 0.
\]
\end{theorem}

\begin{remark} 
In view of Lemma \ref{FGHM 2.11} and Lemma \ref{FGHM 2.13}, Theorem \ref{shifted O-seq} is striking because it covers the whole $AV$-sequence without the need of changing an algebra or a subscheme as we move degree by degree. It is perhaps even more striking because nothing in the definition of the $AV$-sequence suggests that it should be a Hilbert function of anything, or that it should abide by the Macaulay bounds. It is a strong restriction on the possible $AV$-sequences. We will see in this section and the next that even more elegant restrictions on $AV$-sequences arise under some additional geometric assumptions.
\end{remark}

Theorem \ref{shifted O-seq} has the following very surprising consequence about the absence of unexpectedness, giving a partial answer to the second question posed at the beginning of this section. Its proof makes use of the generic initial ideal.

\begin{corollary}[{\cite[Corollary 3.5]{FGHM}}]  \label{cor: a=v}
If $X$ is a reduced subscheme of $\PP^n$ contained in a hypersurface of degree $d+1 \geq 1$ then for any $t \geq d+m$ and $m \geq 1$ we have
\[
\hbox{adim} (X,t,m) = \hbox{vdim}(X,t,m).
\]
In particular, if $X$ is degenerate ($d=0$) then $X$ admits no unexpected hypersurfaces of multiplicity $m$ in any degree $t \geq m$, and hence no unexpected hypersurfaces of any kind.
\end{corollary}

\begin{proof}
Since $I_X$ lies on a hypersurface of degree $d+1$, the initial degree of its homogeneous ideal is $\leq d+1$. In particular, we have $x_0^{d+1} \in \hbox{gin}(I_X)$. But $t-m+1 \geq d+1$ so  $1 \in (\hbox{gin}(I_X) : x_0^{t-m+1})$ for $t \geq d+m$ and $m \geq 0$. 

Now set $t = m+j$ (preparing to use the definition of $AV$-sequences) with $j \geq d$ and $m \geq 1$. This means $j+1 = t-m+1$, and so $1 \in J$, where $J$ is the ideal in Theorem \ref{shifted O-seq}. Thus for $m-1 \geq 0$ we have
\[
0 = h_{R/J}(m-1) = AV_{X,j}(m) = \hbox{adim}(X,t,m) - \hbox{vdim}(X,t,m).
\]
The condition $d=0$ means that for all $t \geq m \geq 1$ the actual dimension and the virtual dimension coincide, hence there is no unexpected hypersurface.
\end{proof}

Here are two more surprising results that force the absence of unexpected hypersurfaces and so address Question \ref{AV question}(2).

\begin{proposition}[{\cite[Proposition 3.8]{FGHM}}]
Let $X \subset \PP^n$ be a closed subscheme. Let $\alpha = \alpha(I_X)$ be the least degree $t$ such that $[I_X]_t \neq 0$. If $AV_{X,0}(\alpha) = 0$ then $X$ does not admit unexpected hypersurfaces, for any degree and multiplicity at a general point.
\end{proposition}

\begin{proposition}[{\cite[Proposition 3.9]{FGHM}}]
Let $X \subset \PP^n$ be a closed subscheme. 
If $\hbox{gin}(I_X)$ is a lex-segment ideal then $X$ does not admit any unexpected hypersurfaces.
\end{proposition}

\noindent The proofs of the above two propositions are technical and are omitted here.


\section{AV sequences for specific classes of subschemes of $\PP^n$}\label{juan2}

Many interesting kinds of behavior have been found, both theoretically and experimentally, for $AV$-sequences in different geometric settings. Here we review some of these, including a very intriguing conjecture. It is somewhat surprising that such nice formulas follow from the original definition of $AV$-sequences and that the $AV$-sequences reflect the geometry of the subscheme in such a strong way.

The first result involves curves in $\PP^3$ and unexpected cones (i.e. $j = t-m = 0$). We have seen in previous sections the importance of unexpected cones in particular. The fact that $AV_{C,0}(t)$ is constant in the next result should be interpreted as meaning that the existence (or not) of unexpected cones persists over all degrees $t \geq e$.

\begin{theorem}[{\cite[Theorem 6.2]{FGHM}}] 
Let $C \subset \PP^3$ be a reduced, equidimensional curve of degree $e$ and arithmetic genus $g$. Then
\[
AV_{C,0}(t) = \binom{e-1}{2} -g
\]
for all $t \geq e$. Moreover, $AV_{C,0}(e) =0$ if and only if $AV_{C,0}(t) = 0$ for $t \geq e$ if and only if $C$ lies in a plane.
\end{theorem}

\begin{proof}
Note that this fits well with Corollary \ref{cor: a=v}, and that there are no cones of degree less than $e$ by B\'ezout's theorem. For the virtual dimension we have
\[
\begin{array}{rcl}
\dim [I_C]_t - \binom{t+2}{3} & = & \binom{t+3}{3} - h_C(t) - \binom{t+2}{3} \\
& = & \binom{t+2}{2} - [te - g + 1]
\end{array}
\]
where $h_C$ is the Hilbert polynomial of $C$ and the fact that for $t \geq e$ we have $h_C(t) = te - g +1$ follows from a result of Gruson, Lazarsfeld and Peskine \cite{GLP}. By B\'ezout one checks that if $S$ is a cone of degree $t$ with vertex $P$ and containing $C$ then $S$ contains as a component the cone over $C$ with vertex $P$. Then the actual dimension is the (vector space) dimension of the linear system of plane curves of degree $t-e$, i.e. it is $\binom{t-e+2}{2}$. Thus it follows after a calculation that
\[
AV_{C,0}(t) = \binom{e-1}{2} - g
\]
and the last part of the theorem follows from the fact that since $\deg C = e$,  $g= \binom{e-1}{2}$ if and only if $C$ is a plane curve. (See {\cite[Proposition 2.1]{HMNT}}.)
\end{proof}

The next result moves in the direction of finite sets of points, which we have studied in the earlier sections of this paper, and again shows the influence of the geometry of the subschemes. We state it without proof.

\begin{theorem}[{\cite[Theorem 6.4]{FGHM}}]  \label{curve, pts}
Let $X \subset \PP^3$ be a finite set of points. Let $C \subset \PP^3$ be a reduced, equidimensional curve of degree $e$ and arithmetic genus $g$. Assume that $X$ is disjoint from $C$. Let $t$ be the smallest integer such that

\begin{itemize}
    \item[(i)] $|X| \leq \binom{t+2}{2}$; and
    \item[(ii)] $X$ imposes independent conditions on forms of degree $t$.
\end{itemize}

\noindent Then
\[
AV_{X\cup C,0} (t+e) = AV_{X,0}(t) + \left [ \binom{e-1}{2} -g \right ].
\]
\end{theorem}

Theorem \ref{curve, pts} has the following consequence for unexpectedness.

\begin{corollary}
Let $X \subset \PP^3$ be a finite set of points. Let $C$ be a reduced plane curve in $\PP^3$ of degree $e$, disjoint from $X$. Then $X$ has an unexpected cone of degree $t$ if and only if $X \cup C$ has an unexpected cone of degree $t + e$. Furthermore, $AV_{X,0} (t) = AV_{X\cup C,0} (t + e)$ for all $t$.
\end{corollary}

\begin{proof}
This is immediate since the arithmetic genus of a plane curve of degree $e$ is $\binom{e-1}{2}$.
\end{proof}

One of the most intriguing examples of $AV$-sequences is actually still partly a conjecture, supported by many experiments and some partial results. We first recall a definition.

\begin{definition}
An {\it SI-sequence} is a finite, non-zero, symmetric $O$-sequence such that the first half is a differentiable $O$-sequence (i.e., also the first difference of the first half is an $O$-sequence). We will refer to an SI-sequence $(1,c,\dots, c,1)$ as a {\it codimension $c$ SI-sequence}.
\end{definition}

The term ``SI-sequence" refers to R. Stanley and A. Iarrobino. The significance of SI-sequences is that they characterize the $h$-vectors of arithmetically Gorenstein subschemes of projective space whose artinian reductions have the Weak Lefschetz Property (WLP). In codimension three they also characterize the $h$-vectors of {\it all} arithmetically Gorenstein subschemes, even though it is only conjectured that all artinian Gorenstein quotients of $k[x,y,z]$ have the WLP. See \cite{BMMNZ} for partial results on this conjecture. Note that SI-sequences are automatically unimodal. See \cite{harima} for additional properties of SI-sequences. 

\begin{conjecture} [{\cite[Conjecture 5.1]{FGHM}}]
Let $X \subset \PP^3$ be a smooth ACM curve not lying on a quadric surface. Then the sequence $AV_{X,1}$ is an SI-sequence (shifted by 1). The last non-zero term in this sequence is $AV_{X,1} (deg X - 5)$, so the SI-sequence ends in degree $\deg X - 6$.
\end{conjecture}

Here are some remarks concerning this conjecture. See \cite[Section 5]{FGHM}  for further details. 

\begin{enumerate}
    \item Despite the strong connection between SI-sequences and artinian Gorenstein algebras, we have not yet seen a direct connection between the setting of our conjecture and Gorenstein algebras. This is part of the allure of this conjecture. 
    
    \item We do not claim that the SI-sequence is necessarily of any given codimension, and indeed we have examples where it has the form $(1,2,\dots)$, $(1,3,\dots)$ and $(1,4,\dots)$. 
    
    \item We do not know if {\it all} SI-sequences (or possibly all SI-sequences $(1,c,\dots)$ with $c \leq 4$) arise as an AV-sequence in this way.

    \item The conjecture is false if we do not have $X$ a smooth ACM curve in $\PP^3$. We have counterexamples if
    
    \begin{itemize}
        \item $X$ is a non-ACM curve in $\PP^3$;
        
        \item $X$ is an ACM curve in $\PP^4$;
        
        \item $X$ is the union of an ACM curve in $\PP^3$ and a finite set of points (even just one point);
        
        \item we consider the AV-sequence $AV_{X,0}$ instead of $AV_{X,1}$;
        
        \item $X$ is zero-dimensional;
        
        \item $X$ is ACM but not irreducible;
        
        \item $X$ is ACM and irreducible but not smooth.
    \end{itemize}
    
    \item If $X$ is an irreducible ACM curve in $\PP^3$ lying on a quadric surface, then  $AV_{X,1}(m)$ is zero. This follows from Corollary \ref{cor: a=v} taking $d=1$ so $j = t-m \geq 1$ -- note that the conjecture is for $j=1$.
    
    \item In \cite[Theorem 5.4]{FGHM}  we show that if $X \subset \PP^3$ is an irreducible ACM curve not lying on a quadric surface then the sequence $AV_{X,1} (m)$ is non-zero and unimodal. Furthermore, the increasing part is a differentiable $O$-sequence. The proof is long and technical, and is omitted here. This result proves parts of the conjecture.
    
\end{enumerate}

Finally, we also have a result for certain curves in $\PP^n$ and their unexpected cones, in terms of $AV$-sequences.

\begin{theorem} [{\cite[Theorem 7.25 and Corollary 7.26]{CFFHMSS}}] 
Let $A$ be the set of coordinate points of $\PP^n$, $|A| = n + 1$. Let $C$ be the union of the $\binom{n+1}{2}$ lines spanned by the pairs of points of $A$. Then the AV-sequence for $C$ measuring unexpected cones is
\[
AV_{C,0}(m) = n + 1
\]
for all $m \gg 0$. In particular, for $m \gg 0$ the dimension of the linear system of cones of degree $m$ with general
vertex $P$ containing $C$ is at least $n + 1$ more than one expects.

If $n \geq 5$ then $C$ has unexpected cones of all degrees $m \geq 3$. Moreover, the dimension of the linear system of cones of degree $m$ with general vertex $P$ containing $C$ exceeds the expected dimension by: 

\begin{itemize} 
\item at least $m$ for all $3 \leq m \leq n + 1$; 
\item at least $n+1$ for all $m \geq n+1$; 
\item and exactly $n+1$ for $m \gg 0$.
\end{itemize}
\end{theorem}



\end{document}